\documentclass[11pt,a4paper]{amsart}
\usepackage{amsmath, amscd, amssymb, amsfonts,amsthm,enumerate}

\usepackage[utf8]{inputenc}

\usepackage{amssymb,color}
\usepackage{amsfonts}
\usepackage{amsmath}
\usepackage{euscript}
\usepackage{enumerate}
\usepackage{graphics}
\usepackage{graphicx}
\usepackage[all,cmtip]{xy}
\usepackage{tikz}
\usetikzlibrary{arrows}

\usepackage{hyperref}
\usepackage{pdfsync}
\newtheorem{theorem}{Theorem}[section]
\newtheorem{proposition}[theorem]{Proposition}
\newtheorem{corollary}[theorem]{Corollary}
\newtheorem{lemma}[theorem]{Lemma}

\theoremstyle{definition}
\newtheorem{definition}[theorem]{Definition}

\newtheorem{conj}[theorem]{Conjecture}
\numberwithin{equation}{section}

\begin{document}

\newcommand{\R}{\mathds{R}}
\newcommand{\Q}{\mathds{Q}}
\newcommand{\cok}{Coker}
\newcommand{\Rt}{\mbox{\newmatha R}}
\newcommand{\Rs}{\mbox{\newmathb R}}
\newcommand{\N}{\mathbb{N}}
\newcommand{\Z}{\mathbb{Z}}
\newcommand{\C}{\mathbb{C}}
\newcommand{\E}{\mathcal{E}}
\newcommand{\G}{\mathcal{G}}
\newcommand{\g}{\mathfrak{g}}
\newcommand{\s}{\mathfrak{sl}}
\newcommand{\A}{\mathcal{A}}
\newcommand{\F}{\mathbb{F}}
\renewcommand{\sl}{\mathfrak{sl}}
\newcommand{\gl}{\mathfrak{gl}}
\newcommand{\cad}{\ar@{}[dr]|{\circlearrowleft}}
\newcommand{\caad}{\ar@{}[ddr]|{\circlearrowleft}}
\newcommand{\cadd}{\ar@{}[drr]|{\circlearrowleft}}
\newcommand{\cai}{\ar@{}[dl]|{\circlearrowleft}}
\newcommand{\verteq}{\rotatebox{90}{$\,=$}}

%\newtheorem{defi}{Definition}[section]
%\newtheorem{lem}[defi]{Lemma}
%\newtheorem{teo}[defi]{Theorem}
%\newtheorem{pro}[defi]{Proposition}
%\newtheorem{cor}[defi]{Corollary}
%\theoremstyle{definition}
%\newtheorem{nota}[defi]{Note}
%\newtheorem{ejemplo}[defi]{Examples}
%\newtheorem{teor}{Teorema}

%%%%%%%%%%%%%%%%%%%%%%%%%%%%%
%%%%%%%%%%%%%%%%%%%%%%%%%%%5
%%%%%%%%%%%%%%%%%%%%%%%%%%%%%%%5
%%%%%%%%%%%%%%%%%%%%%%%%%%%%

\title[Uniserial representations of the Lie algebra
$\sl(2)\ltimes \mathfrak{h}_n$]{Tensor products and intertwining operators between 
two uniserial representations of the Galilean Lie algebra
$\sl(2)\ltimes \mathfrak{h}_n$}

\author{Leandro Cagliero}
\address{FaMAF-CIEM (CONICET), Universidad Nacional de C\'ordoba,
Medina Allende s/n, Ciudad Universitaria, 5000 C\'ordoba, Rep\'ublica Argentina.}
\email{cagliero@famaf.unc.edu.ar}

\author{Iv\'an G\'omez Rivera}
\address{FaMAF-CIEM (CONICET), Universidad Nacional de C\'ordoba,
Medina Allende s/n, Ciudad Universitaria, 5000 C\'ordoba, Rep\'ublica Argentina.}
\email{ivan.gomez.rivera@mi.unc.edu.ar}

\thanks{This research was partially supported by a CONICET grant
PIP 11220210100597CO, SeCyT-UNC grant 33620180100983CB}



\begin{abstract} 
Let $\mathfrak{sl}(2)\ltimes \mathfrak{h}_n$, $n\ge 1$, 
be the Galilean Lie algebra over a field of characteristic zero, 
here $\mathfrak{h}_{n}$ is the Heisenberg Lie algebra of dimension $2n+1$, and $\mathfrak{sl}(2)$ acts on $\mathfrak{h}_{n}$ so that 
$\mathfrak{h}_n\simeq V(2n-1)\oplus V(0)$ as $\mathfrak{sl}(2)$-modules
(here $V(k)$ denotes the irreducible $\mathfrak{sl}(2)$-module of highest weight $k$). In this paper, we study the tensor product of two uniserial representations of 
 $\mathfrak{sl}(2)\ltimes \mathfrak{h}_n$.

We obtain the $\mathfrak{sl}(2)$-module structure of the 
 socle of $V\otimes W$ 
 and we describe the space of intertwining operators 
 $\text{Hom}_{\mathfrak{sl}(2)\ltimes \mathfrak{h}_n}(V,W)$,
 where $V$ and $W$ are uniserial representations of 
 $\mathfrak{sl}(2)\ltimes \mathfrak{h}_n$.
 The structure of the radical of $V\otimes W$
 follows from that of the socle of $V^*\otimes W^*$.

 The result is subtle and shows how difficult is to obtain 
 the whole socle series of arbitrary tensor products of uniserials. In contrast to the serial associative case, our results for 
 $\mathfrak{sl}(2)\ltimes \mathfrak{h}_n$
 reveal that these tensor products are far from being a direct sum of uniserials; in particular, there are cases in which the tensor product of 
two uniserial $\big(\mathfrak{sl}(2)\ltimes \mathfrak{h}_n\big)$-modules is indecomposable but not uniserial. 
Recall that a foundational result of T. Nakayama states that every finitely generated module over a serial associative algebra is a direct sum of uniserial modules. 

This article extends a previous work in which we obtained 
the corresponding results for the Lie algebra $\mathfrak{sl}(2)\ltimes \mathfrak{a}_m$ where  $\mathfrak{a}_m$ is the abelian Lie algebra 
of dimension $m+1$ and $\mathfrak{sl}(2)$ acts so that 
 $\mathfrak{a}_m\simeq V(m)$ as $\mathfrak{sl}(2)$-modules.
\end{abstract}

\maketitle

\section{Introduction}\label{sec:intro}

This article is part of a project whose general goal 
is to understand to what extent there is a 
category consisting of certain finite-dimensional representations of
(non-semisimple) Lie groups such that, on the one hand, 
it includes many (most) representations that are relevant to problems of interest and, on the other hand, 
is small enough
so that its members can be described and handled in a reasonable simple way. 
Clearly, for semisimple Lie groups, the whole category 
of finite-dimensional representations works thanks to 
the highest weight theorem and Weyl's theorem on complete reducibility. 
In the non-semisimple case, the whole category 
of finite-dimensional representations is wild even for abelian Lie algebras of dimension greater than or equal to $2$ (see \cite{GP}). 
This is also discussed in \cite{Sa} 
for the 3-dimensional euclidean Lie
algebra $\mathfrak{e}(2)$, and in \cite{M} for virtually any complex Lie algebra other than semisimple
or 1-dimensional.

Instead of considering the whole category 
of finite-dimensional representations, we and other authors have been working on the idea of describing or classifying
a special class of indecomposable representations of (non-semisimple) Lie algebras whose members might be used as building blocks for describing more general representations. 
For instance, A. Piard \cite{Pi1} analyzed thoroughly the indecomposable modules $U$,
of the complex Lie algebra $\mathfrak{sl}(2)\ltimes \mathbb{C}^2$,
 such that $U/\text{rad}(U)$ is irreducible. 
 More recently, various families of indecomposable modules over various types of non-semisimple Lie algebras have been constructed and/or 
 classified, see for instance \cite{Ca2, Ca1,CMS, CM, DP, Dd, DKR,J}.
Among these attempts, the class of uniserial representations stands out. 
In the articles \cite{CS_JofAlg, CS_canadian, CS_JofAlgApp, CS_Comm, CGS1, CGS2, CLS,Casati2017IndecomposableMO}
 all finite-dimensional 
uniserial representations have been classified 
for some different families of Lie algebras $\g$. 
These classifications show that the class of uniserial representations of $\g$
is rather small and treatable in the universe of all 
indecomposable modules. 
Also, in \cite{Finis2014}, it is shown how the infinite dimensional uniserial representations of 
certain special linear groups obtained in 
\cite{SinThompson2014}
appear naturally in cohomology spaces.
We think that it is natural to consider 
the class of uniserial representations as building blocks
for an `interesting' category of representations, as described
at the beginning of this paper. 
We recall that, for associative algebras, the class of uniserial modules 
is very relevant, a foundational result here
is due to T. Nakayama \cite{Na} (see also \cite{ASS} or\cite{ARS}) and it states that every finitely generated module over a serial ring is a direct sum of uniserial modules. For more information 
in the associative case we refer the reader mainly to \cite{ASS, ARS, Pu}, and also \cite{BH-Z, H-Z,NGB}.
On the other hand, in the Lie algebra case, very little is known about uniserial representations. 

The idea of considering 
uniserial modules as building blocks of an interesting monoidal category naturally leads to the study of   tensor products of uniserial representations and intertwining operators. 
Additionally, another motivation for this project is to find a better way to thoroughly describe certain cohomology spaces associated to 
an algebra (associative or Lie)
viewed as a module over its entire Lie algebra of derivations, which is, in general, non-semisimple. Therefore, tensor products of 
uniserial modules are of particular relevance to us, especially when the cohomology space possesses a Gerstenhaber or a Poisson structure.

In this article and \cite{CGR} we address the problem of describing
the tensor product of two uniserial $\g$-modules in terms of uniserials for certain families of Lie algebras $\g$.
In contrast to the Nakayama case, these tensor products are not at all a direct sum of uniserials, in particular
there are many cases where the tensor product of 
two uniserial $\g$-modules is an indecomposable $\g$-module but not uniserial.

\subsection{Main results} In this paper, all Lie algebras and representations considered 
are assumed to be finite dimensional over a field $\F$ 
of characteristic zero.
For $n\ge 0$, we denote 
by $\mathfrak{a}_{n}$ the abelian Lie algebra of dimension $n+1$, 
by $\mathfrak{h}_{n}$ the Heisenberg Lie algebra of dimension $2n+1$, 
and 
by $V(n)$ the irreducible $\sl(2)$-module with highest weight $n$ ($\dim V(n)=n+1$).
In addition, for $n\ge 1$, 
$\sl(2)\ltimes \mathfrak{a}_n$ denotes 
the Lie algebra obtained by letting $\sl(2)$ act on $\mathfrak{a}_{n}$
so that 
$\mathfrak{a}_n\simeq V(n)$; and, similarly,
 $\sl(2)\ltimes \mathfrak{h}_n$
 denotes 
the Lie algebra where 
$\mathfrak{h}_n\simeq V(2n-1)\oplus V(0)$ as $\sl(2)$-modules. 
We notice that $\sl(2)\ltimes \mathfrak{a}_{2n-1}$ is isomorphic to 
the quotient
 $\sl(2)\ltimes \mathfrak{h}_n$ mod its 1-dimensional center 
 $\mathfrak{z}\big(\sl(2)\ltimes \mathfrak{h}_n\big)\simeq V(0)$. 
The Lie algebras 
$\sl(2)\ltimes \mathfrak{a}_{2n-1}$ (note that $\dim \mathfrak{a}_{2n-1}$ is even) are known as
conformal Galilei Lie algebras, and the 
1-dimensional central extensions
of them, $\sl(2)\ltimes \mathfrak{h}_n$, are the extended conformal Galilei Lie algebras.
For $n=1$, $\sl(2)\ltimes \mathfrak{h}_n$ is also known as 
Schr\"odinger algebra.
Galilei algebras and their representations
attract considerable attention, see for instance 
\cite{Aizawa2012, LU2014,MASTEROV2024,Somberg2018} and 
references therein.
It is worth mentioning that
$\mathfrak{sp}(2n)$ is contained in $\text{Der}(\mathfrak{h}_n)$ 
and thus, $\sl(2)$ may act on  
$\mathfrak{h}_n$ in many different ways. In our case, $\sl(2)$ acts on $\mathfrak{h}_n$
 as a principal $\mathfrak{s}$-triple in $\mathfrak{sp}(2n)$.

In this work, we obtain $\sl(2)$-module structure of 
the socle 
 of the tensor product $V\otimes W$ and, as an application,  we compute
the space of intertwining
operators $\text{Hom}_{\sl(2)\ltimes \mathfrak{h}_n}(V,W)$
where $V$ and $W$ are 
(almost arbitrary) uniserial representations of 
 $\sl(2)\ltimes \mathfrak{h}_n$.  
 
 To describe more precisely our results, we need to recall
 the classification of all the isomorphism classes of 
 uniserial $\big(\sl(2)\ltimes \mathfrak{h}_n\big)$-modules.
 This was obtained in 
 \cite{CGS2}.
 This classification is reviewed with details in 
 \S\ref{sec.3} and a 
 rough description of it is given below. 
 In what follows, and for the rest of the paper, 
 $m=2n-1$.
 
\begin{enumerate}[$\bullet$]
\item \emph{Non-faithful uniserial $\big(\sl(2)\ltimes \mathfrak{h}_n\big)$-modules.} 
Since 
 $\mathfrak{z}\big(\sl(2)\ltimes \mathfrak{h}_n\big)$ acts  trivially on them, 
 they are in correspondence with the 
 uniserial $\big(\sl(2)\ltimes \mathfrak{a}_{m}\big)$-modules.
 In turn, these were classified in \cite{CS_JofAlg}:
 
\medskip

\begin{enumerate}[--]
\item A general family $E(a,b)$, where $a,b$ are non-negative integers with certain restrictions (depending on $n$). 
The composition length 
of $E(a,b)$ is $2$.

\smallskip

\item A general family $Z(a,\ell)$ and its duals. Here 
$a$ and $\ell$ are a non-negative integers. The composition length 
of $Z(a,\ell)$ is $\ell+1$.

\smallskip

\item Some exceptional modules
with composition lengths $3$ and $4$. 
\end{enumerate} 

\medskip

\noindent
The modules $Z(a,\ell)$ and their duals are referred to as \emph{modules of type $Z$.}

\medskip

\item \emph{Faithful  uniserial $\big(\sl(2)\ltimes \mathfrak{h}_n\big)$-modules}. All of them have composition length $3$.

\medskip

\begin{enumerate}[--]
\item For $n=1$: Two families denoted by $FU_a^+$ and $FU_a^-$, with $a$ an integer that satisfies $a\ge0$ and $a\ge1$, respectively.

\smallskip

\item For $n=2$: Only four equivalence classes, they are denoted by $FU_{(0,3,0)}$, 
$FU_{(1,4,1)}$, 
$FU_{(1,2,1)}$ and 
$FU_{(4,3,4)}$.

\smallskip

\item For any $n\geq 3$: Only three equivalence classes, they are denoted by 
$FU_{(0,m,0)}$, 
$FU_{(1,m+1,1)}$ and 
$FU_{(1,m-1,1)}$ (recall that $m=2n-1$). 
\end{enumerate}

\medskip

\noindent
All faithful uniserial modules (except $FU_{(4,3,4)}$, $n=2$)
 are, in some sense, of a similar type and 
are referred to as \emph{standard faithful modules.}
The $\big(\sl(2)\ltimes \mathfrak{h}_2\big)$-module $FU_{(4,3,4)}$ is quite exceptional.
\end{enumerate}

\medskip

The modules $E(a,b)$ (which have composition length equal to 2) constitute the building blocks of all the other uniserial
modules: all uniserials can be obtained by 
combining the modules $E(a,b)$ in a subtle way governed by the zeros of the 6$j$-symbols (see \cite{CGS1,CS_JofAlg}).
As a consequence, the $\big(\sl(2)\ltimes \mathfrak{h}_n\big)$-module structure of the tensor product
of two uniserial representations of 
 $\sl(2)\ltimes \mathfrak{h}_n$ depends strongly on the 
 $\big(\sl(2)\ltimes \mathfrak{h}_n\big)$-module structure of $E(a,b)\otimes E(c,d)$ which is already quite involved.

In \cite{CGR} we stated a conjecture that provides the description of the socle of $E(a,b)\otimes E(c,d)$ for any $a,b,c,d$ (see Conjecture \ref{conj:length2} below) and we proved the part of it
that was necessary to obtain the socle of 
$V\otimes W$ 
and the intertwining
operators $\text{Hom}_{\sl(2)\ltimes \mathfrak{h}_n}(V,W)$
where $V$ and $W$ are 
uniserial representations of 
 $\sl(2)\ltimes \mathfrak{h}_n$ 
 of type $Z$ 
 (in fact, we dealt in \cite{CGR}
 with 
 uniserial representations of
 $\sl(2)\ltimes \mathfrak{a}_m$, instead of 
 $\sl(2)\ltimes \mathfrak{h}_n$, recall that,
 for modules of type $Z$,
 the action of $\mathfrak{z}\big(\sl(2)\ltimes \mathfrak{h}_n\big)$ is trivial). 
 In particular, we proved that  the socle of 
$V\otimes W$ is multiplicity free as $\sl(2)$-modules.
As an application of these results, we proved in \cite{CGR} that if $V$ and $W$ are 
 $\big(\sl(2)\ltimes \mathfrak{h}_n\big)$-modules 
 of type $Z$, then $V$ and $W$ are determined from $V\otimes W$.
 Moreover, we 
 provided a procedure to identify the corresponding 
 parameters $a$ and $\ell$ of $V$ and $W$ from $V\otimes W$. 
 This is a rare property, even if the factors are irreducible, it is not frequent that the factors $V$ and $W$  are determined
 from $V\otimes W$ (see \cite{MOROTTI_Rep_Theory} and references within).

 \medskip

In this paper, we extend the results of 
 \cite{CGR} obtaining 
 the socle of 
$V\otimes W$ 
and the intertwining
operators $\text{Hom}_{\sl(2)\ltimes \mathfrak{h}_n}(V,W)$
when both $V$ and $W$ are 
standard faithful uniserial $\big(\sl(2)\ltimes \mathfrak{h}_n\big)$-modules, 
or when one of them is standard faithful and the other one is uniserial of type $Z$.
 In contrast to the non-faithful case, in the standard faithful case it may happen that 
 the socle of 
 $V\otimes W$ is not multiplicity free as $\sl(2)$-module
 (this occurs when $V\simeq W$). 
 As a consequence, if 
$V$ and $W$ are isomorphic standard faithful uniserials, 
then the space of 
intertwining operators $\text{Hom}_{\sl(2)\ltimes \mathfrak{h}_n}(V,W)$
is 2-dimensional.
 The main step toward these results 
 requires making a considerable advance 
 in the proof of Conjecture \ref{conj:length2}.
 See the comments after it to know what is
 still open  about this conjecture.

 \medskip

The paper is organized as follows. 
 In \S\ref{Sec.Preliminaries} we review some basic facts about 
 uniserial representations of Lie algebras and recall all the 
 necessary definitions and formulas involving 
 the Clebsch-Gordan coefficients. 
In \S\ref{sec.3} we review the classification of all 
uniserial representations of the Lie algebras 
$\sl(2)\ltimes \mathfrak{a}_{m}$ (obtained in \cite{CS_JofAlg})
and 
$\sl(2)\ltimes \mathfrak{h}_{n}$ (obtained in \cite{CGS2}).
The main section of the paper is \S\ref{sec.4} and we obtain in it the 
$\sl(2)$-module structure of the socle of the tensor product of
two (non-exceptional) uniserial $\big(\sl(2)\ltimes \mathfrak{h}_n\big)$-modules
$V$ and $W$:
Theorem \ref{thm:main} recalls the case when 
 $V$ and $W$ are of type $Z$ (obtained in \cite{CGR}), 
 Conjecture \ref{conj:length2} deals with all the possible cases of
 composition length 2, 
 Theorem \ref{prop.conjecture} confirms the part of Conjecture \ref{conj:length2} needed to 
 prove Theorems \ref{thm.fielvsnofiel} and \ref{thm.fielvsfiel},
 Theorem \ref{thm.fielvsnofiel} gives the socle when $V$ is of type $Z$ and 
 $W$ is standard faithful, and 
 Theorem \ref{thm.fielvsfiel} describes the socle when both $V$ and 
 $W$ are standard faithful. 
 Finally, in \S\ref{sec.5} we obtain the space of intertwining operators as an application of the results obtained in \S\ref{sec.4}. 
 Our proof of 
 Theorem \ref{prop.conjecture} is technical and long, 
 it requires to consider 
 some linear systems with entries given by the 
 Clebsch-Gordan coefficients, and thus we decided to devote \S\ref{sec.proof_conj} to it.

\section{Preliminaries}
\label{Sec.Preliminaries}
\subsection{The Clebsch-Gordan coefficients}
\label{Subsec.Clebsch-Gordan}
Recall that $\F$ is a field of characteristic zero and that all Lie algebras and representations are assumed to be finite dimensional 
over $\F$.
Let 
\begin{equation}\label{eq.basis_sl2}
e=\begin{pmatrix}
0 & 1 \\
0 & 0
\end{pmatrix},\qquad
h=\begin{pmatrix}
1 & 0 \\
0 & -1
\end{pmatrix},\qquad
f=\begin{pmatrix}
0 & 0 \\
1 & 0
\end{pmatrix}
\end{equation}
be the standard basis of $\sl(2)$.
Let $V(a)$ be the irreducible $\sl(2)$-module with highest weight $a\ge0$.
We fix a basis $\{v_0^a,\dots,v_a^a\}$ of $V(a)$ relative to which the basis $\{e,h,f\}$ acts as follows:
\begin{align}\notag
e\, v_k^{a}=&\sqrt{
\frac{a}{2}\left(\frac{a}{2}+1\right)-
\left(\frac{a}{2}-k+1\right)\left(\frac{a}{2}-k\right)
}
v_{k-1}^{a},\\[2mm]\label{eq.sl2-action_Vm}
h\, v_k^{a}=&(a-2k)v_k^{a},\\[2mm]\notag
f\, v_k^{a}=&\sqrt{
\frac{a}{2}\left(\frac{a}{2}+1\right)-
\left(\frac{a}{2}-k-1\right)\left(\frac{a}{2}-k\right)
}
v_{k+1}^{a},
\end{align} 
where $0\leq k\leq a$ and $v_{-1}^a=v_{a+1}^a=0$.
The basis $\{v_0^a,\dots,v_a^a\}$ has been chosen in a convenient way to
introduce below the Clebsch-Gordan coefficients.
Note that if we denote by $(x)_a$ the matrix of 
$x\in\sl(2)$ relative to the basis $\{v_0^a,\dots,v_a^a\}$, then 
$\{(e)_1,(h)_1,(f)_1\}$
are as in \eqref{eq.basis_sl2},
and 
\begin{equation*}
(e)_2=\begin{pmatrix}
0 & \sqrt{2} & 0 \\
0 & 0 & \sqrt{2} \\
0 & 0 & 0
\end{pmatrix},\qquad
(h)_2=\begin{pmatrix}
2 & 0 & 0 \\
0 & 0 & 0 \\
0 & 0 & -2
\end{pmatrix},\qquad
(f)_2=\begin{pmatrix}
0 & 0 & 0 \\
\sqrt{2} & 0 & 0 \\
0 & \sqrt{2} & 0
\end{pmatrix}. 
\end{equation*}
This means 
that we may assume that 
$\{v_0^2,v_1^2,v_2^2\}=\{-e,\frac{\sqrt{2}}{2}h,f\}$.

We know that $V(a)\simeq V(a)^*$ as $\sl(2)$-modules. 
More precisely, if $\{(v_0^a)^*,\dots,(v_a^a)^*\}$ 
is the dual basis of $\{v_0^a,\dots,v_a^a\}$ then 
the map 
\begin{equation}\label{eq.dual}
\begin{split}
V(a) & \rightarrow V(a)^* \\
v_k^a & \mapsto (-1)^{a-k} (v_{a-k}^a)^* 
\end{split}
\end{equation}
gives an explicit $\sl(2)$-isomorphism.

It is well known that the decomposition of the tensor product
of two irreducible $\sl(2)$-modules $V(a)$ and $V(b)$
 is 
\begin{equation}\label{eq.tensor}
V(a)\otimes V(b)\simeq V(a+b)\oplus V(a+b-2) \oplus \cdots \oplus V(|a-b|).
\end{equation}
This is the well known Clebsch-Gordan formula.

The 
\emph{Clebsch-Gordan coefficients}
\[
CG(j_{1},m_{1};j_{2},m_{2}\mid j_3,m_3)
\]
are defined below and they 
provide an explicit $\sl(2)$-embedding 
$V(c) \rightarrow V(a)\otimes V(b)$ 
which is the following
\begin{align*}
V(c) & \rightarrow V(a)\otimes V(b) \\
v_k^c & \mapsto v_k^{a,b,c}
\end{align*}
where, by definition, 
\begin{equation}\label{eq.Vc_en_tensor}
v_k^{a,b,c}=\sum_{i,j} 
CG(\tfrac{a}{2},\tfrac{a}{2}-i;\,\tfrac{b}{2},\tfrac{b}{2}-j
\,|\,\tfrac{c}{2},\tfrac{c}{2}-k)\,
 v_i^a\otimes v_j^b,
\end{equation}
where the sum runs over all $i,j$ such that 
$\tfrac{a}{2}-i+\tfrac{b}{2}-j=\tfrac{c}{2}-k$
(in fact, we could let $i,j$ run freely since the Clebsch-Gordan
coefficient involved is zero if $\tfrac{a}{2}-i+\tfrac{b}{2}-j\ne\tfrac{c}{2}-k$).
Since 
\begin{equation}\label{eq.Hom}
\text{Hom}(V(b),V(a))\simeq V(b)^*\otimes V(a) 
\simeq V(a)\otimes V(b) 
\end{equation}
it follows from \eqref{eq.dual} and \eqref{eq.Vc_en_tensor} that 
the map $V(c) \rightarrow \text{Hom}(V(b),V(a))$
given by 
\begin{align}
v_k^c & \mapsto 
\sum_{i,j} 
CG(\tfrac{a}{2},\tfrac{a}{2}-i;\,\tfrac{b}{2},\tfrac{b}{2}-j
\,|\,\tfrac{c}{2},\tfrac{c}{2}-k)\,
 v_i^a\otimes v_j^b, \notag \\
 & \mapsto 
\sum_{i,j} (-1)^{b-j}
CG(\tfrac{a}{2},\tfrac{a}{2}-i;\,\tfrac{b}{2},\tfrac{b}{2}-j
\,|\,\tfrac{c}{2},\tfrac{c}{2}-k)\,
 v_i^a\otimes (v_{b-j}^b)^*, \notag \\ \label{eq.embedding_Hom}
 & \mapsto 
\sum_{i,j} (-1)^{j}
CG(\tfrac{a}{2},\tfrac{a}{2}-i;\,\tfrac{b}{2},-\tfrac{b}{2}+j
\,|\,\tfrac{c}{2},\tfrac{c}{2}-k)\,\,
(v_{j}^b)^* \otimes v_i^a
\end{align}
is an $\sl(2)$-module homomorphism.

We now recall briefly the basic definitions and 
facts about the Clebsch-Gordan coefficients.
We will mainly follow \cite{VMK}.

Given three non-negative integers or half-integers $j_1,j_2,j_3$, we say that they \emph{satisfy
the triangle condition} if
$j_1+j_2+j_3$ is an integer and 
they can be the side lengths of a (possibly degenerate) 
triangle (that is
$|j_1-j_2|\le j_3\le j_1+j_2$).
We now define (see \cite[\S8.2, eq.(1)]{VMK})
\[
 \Delta(j_1,j_2,j_3)=\sqrt{\frac{(j_1+j_2-j_3)!(j_1-j_2+j_3)!(-j_1+j_2+j_3)!}{(j_1+j_2+j_3+1)!}}
 \]
if $j_1, j_2, j_3$ satisfies the triangle condition;
otherwise, we set $\Delta(j_1,j_2,j_3)=0$.

If, in addition, $m_1$, $m_2$ and $m_3$ are three integers or half-integers, then 
the corresponding \emph{Clebsch-Gordan coefficient}
\[ 
CG(j_{1},m_{1};j_{2},m_{2}| j_3,m_3)
\]
is zero unless $m_1+m_2= m_3$ and $|m_i|\le j_i$ for $i=1,2,3$. In this case, the following formula is valid for $m_3\ge 0$ and $j_1\ge j_2$ (see \cite[\S8.2, eq.(3)]{VMK})
 \begin{multline*}
 CG(j_{1},m_{1};j_{2},m_{2}\mid j_3,m_3)=
 \Delta(j_1,j_2,j_3)\,\sqrt{(2j_3+1) } \\[1mm]
 \times \sqrt{(j_1+m_1)!(j_1-m_1)!(j_2+m_2)!(j_2-m_2)!(j_3+m_3)!(j_3-m_3)! } \\[1mm]
 \times
 \sum_r\frac{(-1)^r}{r!(j_1\!+\!j_2\!-\!j_3\!-\!r)!(j_1\!-\!m_1\!-\!r)!(j_2\!+\!m_2\!-\!r)!(j_3\!-\!j_2\!+\!m_1\!+\!r)!(j_3\!-\!j_1\!-\!m_2\!+\!r)!},
 \end{multline*}
 where the sum runs through all integers
$r$ for which the argument of every factorial is non-negative.
If either $m_3< 0$ or $j_1< j_2$ 
we have
\begin{align}
 CG(j_{1},m_{1};j_{2},m_{2}\mid j_3,m_3)
 & = (-1)^{j_1+j_2-j_3}\;
 CG(j_{1},-m_{1};j_{2},-m_{2}\mid j_3,-m_3)\notag
 \\[2mm]\label{eq:swap}
 & = (-1)^{j_1+j_2-j_3}\;
 CG(j_{2},m_{2};j_{1},m_{1}\mid j_3,m_3). 
\end{align}
In addition, it also holds
\begin{equation}\label{eq.simmetryCG}
 CG(j_{1},m_{1};j_{2},m_{2}\mid j_3,m_3)
=(-1)^{j_1-m_1}
\sqrt{\frac{2j_3+1}{2j_2+1}}\;
 CG(j_{1},m_{1};j_{3},-m_{3}\mid j_2,-m_2).
\end{equation}
In the following sections, we will need the following particular values of the Clebsch-Gordan coefficients.
Here, $a,b$ are integers and $i=0,\dots,a$, 
$j=0,\dots,b$. 

\begin{equation}\label{eq.extremoCG0}
CG(\tfrac{a}{2},\tfrac{a}{2}-i;\,\tfrac{b}{2},\tfrac{b}{2}-j
\,|\,\tfrac{a+b}{2},\tfrac{a+b}{2}-i-j)
=
\sqrt{\frac{a!b!(a+b-i-j)!(i+j)!}{i!j!(a+b)!(a-i)!(b-j)!}},
\end{equation}
\begin{multline}\label{eq.extremoCG1}
CG(\tfrac{a}{2},\tfrac{a}{2}-i;\,\tfrac{b}{2},j-\tfrac{b}{2}
\,|\,\tfrac{a-b}{2},\tfrac{a-b}{2}-i+j) \\
=(-1)^j
\sqrt{\frac{(a-i)!\;i!\;b!\;(a-b+1)!}
{(a+1)!\;j!\;(b-j)!\;(a-b-i+j)!\;(i-j)!}},
\end{multline}
\begin{multline}\label{eq.extremoCG11}
CG(\tfrac{a}{2},i-\tfrac{a}{2};\,\tfrac{b}{2},\tfrac{b}{2}-j
\,|\,\tfrac{b-a}{2},\tfrac{b-a}{2}+i-j) \\
=(-1)^{a}CG(\tfrac{b}{2},\tfrac{b}{2}-j;\,
\tfrac{a}{2},i-\tfrac{a}{2}
\,|\,\tfrac{b-a}{2},\tfrac{b-a}{2}+i-j) \\
=(-1)^j
\sqrt{\frac{(b-j)!\;j!\;a!\;(b-a+1)!}
{(b+1)!\;i!\;(a-i)!\;(b-a-j+i)!\;(j-i)!}},
\end{multline}
\begin{multline}\label{eq.extremoCG}
CG(
\tfrac{a}{2},\tfrac{a}{2}-i;\,
\tfrac{b}{2},\tfrac{b}{2}-j\,|\,
\tfrac{a+b}{2}-i-j,\tfrac{a+b}{2}-i-j) \\
=
(-1)^i
\sqrt{\frac{(a+b-2i-2j+1)!\;(i+j)!\;(a-i)!\;(b-j)!}
{(a+b-i-j+1)!\;(a-i-j)!\;(b-i-j)!\;i!\;j!}}.
\end{multline}

\subsection{Uniserial representations}

Given a Lie algebra $\g$, a $\g$-module $V$ is 
\emph{uniserial} if it admits a unique composition series. 
In other words, $V$ is uniserial if the socle series
\[
0 = \text{soc}^0(V)\subset \text{soc}^1(V) \subset \cdots \subset \text{soc}^n(V) = V
\]
is a composition series of $V$, that is, the socle factors $\text{soc}^{i}(V)/\text{soc}^{i-1}(V)$ are irreducible for all $1\le i \le n$. 
Recall that $\text{soc}^1(V)=\text{soc}(V)$ 
is the sum of all irreducible  
$\g$-submodules of $V$ and 
$\text{soc}^{i}(V)/\text{soc}^{i-1}(V)
=\text{soc}(V/\text{soc}^{i-1}(V))$.
Note that for uniserial modules, 
the \emph{composition length} of $V$ coincides with
 its \emph{socle length}.

If the
 Levi decomposition of $\g$ is $\g = \mathfrak{s} \ltimes \mathfrak{r}$, (with $\mathfrak{r}$ the solvable radical and $\mathfrak{s}$ semisimple) we may choose irreducible 
 $\mathfrak{s}$-submodules $V_i\subset V$, $1\le i \le n$, such that 
 \begin{equation}\label{eq.soc_decomp}
 V=V_1\oplus \cdots \oplus V_n 
 \end{equation}
with 
$V_i
 \simeq \text{soc}^{i}(V)/\text{soc}^{i-1}(V)$ 
 as  $\mathfrak{s}$-modules and 
 \[
 \mathfrak{r} V_i \subset V_1\oplus \cdots \oplus V_i. 
 \]
In fact, if 
$[\mathfrak{s},\mathfrak{r}]=\mathfrak{r}$, then $\mathfrak{r} V_i \subset V_1\oplus \cdots \oplus V_{i-1}$, 
see Lemma \ref{lemma:soc} below. 

\begin{definition}\label{rmk:order}
We say that \eqref{eq.soc_decomp} is the 
\emph{socle decomposition} of $V$.
We point out that, in the socle decomposition of a 
$\g$-module, the order of the summands is relevant.
\end{definition}

The proof of the following lemma can be found in 
\cite[Lemmas 2.1 and 2.2]{CGR}.

\begin{lemma}\label{lemma:soc}
Assume that $\mathfrak{r}=[\mathfrak{s},\mathfrak{r}]$ and 
let $V$ be a $\g$-module.
Then 
\begin{enumerate}[(1)]
\item $\text{soc}(V)=V^\mathfrak{r}$.
\item If $V=V_1\oplus \cdots \oplus V_n$ is a vector space
 decomposition such that $\text{soc}(V)=V_1$ and 
$\mathfrak{r} V_k \subset V_{k-1}$ for all $k=2,\dots,n$, then 
$\text{soc}^k(V)=V_1\oplus \cdots \oplus V_k$
for all $k=1,\dots,n$.
\end{enumerate}
\end{lemma}

\section{Uniserial representations of 
\texorpdfstring{$\sl(2)\ltimes \mathfrak{h}_n$}{}}\label{sec.3}

\subsection{The Lie algebra \texorpdfstring{$\sl(2)\ltimes \mathfrak{h}_n$}{}}
Let us introduce some notation, recall some basic facts about 
the Heisenberg Lie algebra $\mathfrak{h}_n$ and define
the Lie algebra $\sl(2)\ltimes \mathfrak{h}_n$.
Let us fix $n\geq 1$ and, for the rest of the paper, set $m=2n-1$.
We know that $\sl(2)$ acts by derivations on $\mathfrak{h}_n$ in such a way that 
\[
\mathfrak{h}_n\simeq V(m)\oplus V(0)
\]
as $\sl(2)$-modules. 
In order to have this action described in harmony with \S\ref{Subsec.Clebsch-Gordan},
we:
\begin{itemize}
\item[\tiny$\bullet$] assume, by definition, that $\mathfrak{h}_n=V(m)\oplus V(0)$, ($m=2n-1$),
\item[\tiny$\bullet$] denote by $\mathfrak{h}_n(m)$ the subspace $V(m)$ of $\mathfrak{h}_n$,
\item[\tiny$\bullet$] denote by 
$
\{e_0,\hdots,e_m\}
$
the basis of $\mathfrak{h}_n(m)$ corresponding to the basis $\{v_0^m,\hdots, v_m^m\}$ of $V(m)$ chosen in \eqref{eq.sl2-action_Vm}
(that is $e_k=v_k^m$ when $V(m)$ is viewed inside $\mathfrak{h}_n$),
\item[\tiny$\bullet$] denote by $z\in\mathfrak{h}_n$ the vector $v_0^0\in V(0)$ chosen in \eqref{eq.sl2-action_Vm}.
\end{itemize}
In this context, the non-zero brackets in $\mathfrak{h}_n$ are
\begin{align}\label{relacion-V(m)-V(0)}
[e_i,e_{m-i}]&=CG(\tfrac{m}{2},\,\tfrac{m}{2}-i;\,\tfrac{m}{2},-\tfrac{m}{2}+i\,|\,0,\,0)\,z \notag\\
&= (-1)^i\,\sqrt{\tfrac{1}{m+1}}\,z.
\end{align}
The center of $\mathfrak{h}_n$
is 
$\mathfrak{z}(\mathfrak{h}_n)=\F z$ and 
$\mathfrak{h}_n/\F z\simeq \mathfrak{a}_{m}$ as Lie algebras.
It is straightforward to see that \eqref{eq.sl2-action_Vm}
provides $\mathfrak{h}_n$ with an action of $\sl(2)$ by derivations. 
This action allow us to define the Lie algebras 
$\sl(2)\ltimes \mathfrak{h}_n$ and 
$\sl(2)\ltimes \mathfrak{a}_m$ 
($\mathfrak{a}_m$ is abelian of dimension $m+1=2n$). 
It is clear that $\F z$ is also the center of $\sl(2)\ltimes \mathfrak{h}_n$ and 
\[
\big(\sl(2)\ltimes \mathfrak{h}_n\big)/\F z\simeq \sl(2)\ltimes \mathfrak{a}_{m}
\]
as Lie algebras. 

In \cite{CGS2}, it is obtained the classification, up to isomorphism, of all uniserial representations of the Lie algebra $\sl(2)\ltimes \mathfrak{h}_n$. 
It is straightforward to see that a uniserial 
representation of $\sl(2)\ltimes \mathfrak{h}_n$ is faithful 
if and only if $z$ acts non-trivially.
Therefore, the classification was given in two stages: the non-faithful and the faithful ones.
The non-faithful ones are the same as those of the Lie algebra $\sl(2)\ltimes \mathfrak{a}_{m}\simeq \sl(2)\ltimes V(m)$ which were classified earlier in
\cite[Theorem 10.1]{CS_JofAlg}.
We now recall this classification.

\subsection{The non-faithful 
\texorpdfstring{$\big(\sl(2)\ltimes \mathfrak{h}_n\big)$-modules $E(a,b)$}{}}
If $a$ and $b$ are non-negative integers such that 
$\frac{m}2,\frac{a}2,\frac{b}2$ satisfy the triangle condition, 
it follows from \eqref{eq.tensor} and \eqref{eq.Hom}
that, up to scalar, there is a unique
$\sl(2)$-module homomorphism 
\[
V(m)\to \text{Hom}(V(b),V(a)).
\] 
Recall that the radical $\mathfrak{r}$ of 
$\sl(2)\ltimes \mathfrak{a}_{m}$ is $\mathfrak{a}_{m}\simeq V(m)$
as $\sl(2)$-modules. 
Thus, the above $\sl(2)$-module homomorphism
produces an action of 
$\mathfrak{r}$ on 
\[
E(a,b)=V(a)\oplus V(b)
\]
 such that 
$\mathfrak{r}$ maps $V(a)$ to $0$ and maps 
 $V(b)$ to $V(a)$ 
 as shown below for the basis 
 $\{e_s:s=0,\dots,m\}$ of $\mathfrak{a}_{m}$:
\begin{equation}\label{eq.actionV(m)}
e_s\, v_j^{b}=\sum_{i=0}^a
(-1)^j\, CG(\tfrac{a}{2},\tfrac{a}{2}-i;\,\tfrac{b}{2},-\tfrac{b}{2}+j
\,|\,\tfrac{m}{2},\tfrac{m}{2}-s)\,
v_{i}^{a}.
\end{equation}
Note that \eqref{eq.actionV(m)} is equivalent to \eqref{eq.embedding_Hom}.
Note also that the above sum has, in fact, at most one summand, that is
\begin{equation}\label{eq.actionV(m)1}
e_s\, v_j^{b}=
\begin{cases}
0,& \text{if $i\ne j+s+\frac{a-b-m}{2}$;} \\[2mm]
(-1)^j CG(\tfrac{a}{2},\tfrac{a}{2}-i;\,\tfrac{b}{2},-\tfrac{b}{2}+j
\,|\,\tfrac{m}{2},\tfrac{m}{2}-s)\,\,v_{i}^{a},& 
\text{if $i=j+s+\frac{a-b-m}{2}$.}
\end{cases}
\end{equation}
This action, combined with the action of $\sl(2)$ defines a uniserial 
$\big(\sl(2)\ltimes \mathfrak{a}_{m}\big)$-module structure with composition length 2 on $E(a,b)=V(a)\oplus V(b)$.
Consequently, this also defines a uniserial 
$\big(\sl(2)\ltimes \mathfrak{h}_{n}\big)$-module structure 
(of composition length 2) on $E(a,b)$ on which the center 
of $\sl(2)\ltimes \mathfrak{h}_{n}$ acts trivially.

In the following example, $n=2$ and thus $m=3$. 
The matrix below shows explicitly the action of $\sl(2)\ltimes \mathfrak{a}_{3}$
on the module $E(3,2)$. Recall that $\{e,h,f,e_s:s=0\dots 3\}$ is a basis of 
$\sl(2)\ltimes \mathfrak{a}_{3}$. 
The letters $e,h,f,e_s$ in the matrix only
indicate the non-zero entries of the corresponding basis element. 
\[
 \left( 
 \begin{array}{rrrrrrr} 
 3\, h & \sqrt{3}\,e & 0\;\; &  0\;\; &
 \sqrt{\frac{2}{5}}\,{e_1} & -\sqrt{\frac{3}{5}}\,{e_0}  & 0\;\;\; \\[2mm] 
 \sqrt{3}\,f & h\; & 2\, e & 0\;\; & \;\;
 2\sqrt{\frac{2}{15}}\,{e_2} & -\sqrt{\frac{1}{15}}\,{e_1} & -\sqrt{\frac{2}{5}}\,{e_0} \\[2mm]
 0\;\; & 2\, f & -h & \sqrt{3}\,e & 
 \sqrt{\frac{2}{5}}\,{e_3} & \sqrt{\frac{1}{15}}{e_2} & -2\sqrt{\frac{2}{15}}\,{e_1} \\[2mm] 
0\;\; & 0\;\; & \sqrt{3}\,f & -3\, h & 
0\;\;\; & \sqrt{\frac{3}{5}}\,{e_3} & -\sqrt{\frac{2}{5}}\,{e_2} \\[4mm] 
0 & 0 & 0 & 0 & 2\, h & \sqrt{2}\,e & 0\;\; \\[2mm] 
0 & 0 & 0 & 0 & \sqrt{2}\,f & 0\;\; & \sqrt{2}\,e \\[2mm] 
0 & 0 & 0 & 0 & 0\;\; & \sqrt{2}\,f & -2\, h
\end {array} \right) 
\]

It is straightforward to see that 
$E(a,b)^*\simeq E(b,a)$.
The action given in \eqref{eq.actionV(m)}
is the main building block for all other uniserial 
$\big(\sl(2)\ltimes \mathfrak{a}_{m}\big)$-modules as follows.

\subsection{Non-faithful 
\texorpdfstring{$\big(\sl(2)\ltimes \mathfrak{h}_n\big)$-modules of type $Z$}{}}
The above construction can be extended to arbitrary composition length 
\[
V(a_0)\oplus V(a_1)\oplus\cdots\oplus V(a_\ell)
\]
only when the sequence $\{a_i\}$ is monotonic (increasing or decreasing) and $|a_i-a_{i-1}|=m$, for all $i=1,\dots,\ell$.
More precisely, for the ``increasing case" 
let 
$\alpha$ and $\ell$ be non-negative integers and let 
$Z(\alpha,\ell)$ be the $\big(\sl(2)\ltimes \mathfrak{a}_{m}\big)$-module defined by
 \begin{equation}\label{eq.soc_decomp_Z}
Z(\alpha,\ell)=V(\alpha)\oplus V(\alpha+m)\oplus \cdots \oplus V(\alpha+\ell m)
 \end{equation}
as $\sl(2)$-module with action of 
$\mathfrak{r}=\mathfrak{a}_m$ sending 
\[
0\longleftarrow
V(\alpha)\longleftarrow 
V(\alpha+m)\longleftarrow \dots\longleftarrow 
V(\alpha+\ell m)
\]
as indicated in \eqref{eq.actionV(m)}: 
that is, the $i^{\text{th}}$-arrow 
($i=1,\dots,\ell$)
\[
V(\alpha+(i-1)m)\longleftarrow V(\alpha+im)
\]
means that the basis $\{e_s:s=0,\dots,m\}$ of $\mathfrak{a}_m$ acts as
\begin{equation*}
e_s\, v_j^{b}=\sum_{i=0}^a
(-1)^j\, CG(\tfrac{a}{2},\tfrac{a}{2}-i;\,\tfrac{b}{2},-\tfrac{b}{2}+j
\,|\,\tfrac{m}{2},\tfrac{m}{2}-s)\,
v_{i}^{a},
\end{equation*}
with $a=\alpha+(i-1)m$ and $b=\alpha+im$. 
This way, $Z(\alpha,\ell)$ becomes a uniserial  
$\big(\sl(2)\ltimes \mathfrak{a}_{m}\big)$-module 
and also a uniserial 
$\big(\sl(2)\ltimes \mathfrak{h}_{n}\big)$-module 
in which the center 
of $\sl(2)\ltimes \mathfrak{h}_{n}$ acts trivially.
(We point out that the above sequence serves as an indication 
of the action of $\mathfrak{r}$, there is no chain complex involved.)

We notice that $Z(\alpha,0)=V(\alpha)$ ($\mathfrak{r}$ acts trivially) and 
$Z(\alpha,1)=E(\alpha,\alpha+m)$, 
as $\big(\sl(2)\ltimes \mathfrak{h}_{n}\big)$-modules. 

The ``decreasing case" corresponds to the dual modules $Z(\alpha,\ell)^*$.
The modules $Z(\alpha,\ell)$ and $Z(\alpha,\ell)^*$
are called \emph{of type $Z$} and they are the unique 
isomorphism classes of uniserial $\big(\sl(2)\ltimes \mathfrak{a}_{m}\big)$-modules of 
composition length $\ell+1$ for $\ell \ge 4$. 
See Theorem \ref{thm.CS_Classification} below and  \cite{CS_JofAlg} for more details and explicit examples.

\subsection{Non-faithful 
\texorpdfstring{$\big(\sl(2)\ltimes \mathfrak{h}_n\big)$-modules of exceptional type (composition lengths $2$, $3$ and $4$)}{}}
The modules $E(a,b)$ with $|a-b|\ne m$ are not of type $Z$ and we consider them of exceptional type (of composition lengths $2$). 
For composition lengths $3$ and $4$ there are very few possible ways to ``combine" the modules $E(a,b)$ so that we do not fall in type $Z$.

For composition length equal to $3$, given $0\le c< 2m$ and 
$c\equiv 2m\mod 4$, 
 let 
\[
E_3(c)=
V(0)\oplus V(m)\oplus V(c)
\]
as $\sl(2)$-modules with action of 
$\mathfrak{r}$ sending 
\begin{center}
\begin{tikzpicture}[->,>=stealth',auto,node distance=2cm,thick]
 \node (0) {$0$};
 \node (1) [right of=0] {$V(0)$};
 \node (2) [right of=1] {$V(m)$};
 \node (3) [right of=2] {$V(c)$};

 \path[every node/.style={font=\sffamily\small}]
  (1) edge node [right] {} (0)
  (2) edge node [right] {} (1)
  (3) edge node [right] {} (2);
\end{tikzpicture}
\end{center}
 with the maps $V(c)\to V(m)$ and $V(m)\to V(0)$
given by \eqref{eq.actionV(m)}.

For composition length equal to $4$, if $m\equiv 0\mod 4$, 
there is a family of $\big(\sl(2)\ltimes \mathfrak{a}_{m}\big)$-modules, parameterized
by a non-zero scalar $t\in\F$, with a fixed socle decomposition. This is defined by 
\[
E_4(t)=
V(0)\oplus V(m)\oplus V(m)\oplus V(0)
\]
as $\sl(2)$-modules with action of 
$\mathfrak{r}$, sending each irreducible component as shown by the arrows
\begin{center}
\begin{tikzpicture}[->,>=stealth',auto,node distance=2cm,thick]
 \node (0) {$0$};
 \node (1) [right of=0] {$V(0)$};
 \node (2) [right of=1] {$V(m)$};
 \node (3) [right of=2] {$V(m)$};
 \node (4) [right of=3] {$V(0)$};

 \path[every node/.style={font=\sffamily\small}]
  (1) edge node [right] {} (0)
  (2) edge node [right] {} (1)
  (3) edge node [right] {} (2)
  (4) edge node [right] {} (3)
  (4) edge[bend right] node [left] {} (2);
\end{tikzpicture}
\end{center}
where the horizontal arrows are 
 given by \eqref{eq.actionV(m)} 
 and the bent arrow is $t$ times \eqref{eq.actionV(m)}.
 Again, these $\big(\sl(2)\ltimes \mathfrak{a}_{m}\big)$-modules
 can be viewed as $\big(\sl(2)\ltimes \mathfrak{h}_{n}\big)$-modules 
on which the center 
of $\sl(2)\ltimes \mathfrak{h}_{n}$ acts trivially. 

\subsection{Classification of all non-faithful 
\texorpdfstring{$\big(\sl(2)\ltimes \mathfrak{h}_n\big)$-modules}{}}
As we said at the beginning of the section, 
a uniserial 
representation of $\sl(2)\ltimes \mathfrak{h}_n$ is faithful 
if and only if $z$ acts non-trivially.
Thus, the 
non-faithful uniserial 
$\big(\sl(2)\ltimes \mathfrak{h}_n\big)$-modules
are in correspondence, via the projection
$\sl(2)\ltimes \mathfrak{h}_n\to \sl(2)\ltimes \mathfrak{a}_{m}$,
 with the uniserial representations 
of $\sl(2)\ltimes \mathfrak{a}_{m}$.
These were classified in \cite[Theorem 10.1]{CS_JofAlg}, 
we summarize that result in the following theorem.

\begin{theorem}\label{thm.CS_Classification}
The following list describes all the isomorphism classes of non-faithful uniserial representations 
of $\sl(2)\ltimes \mathfrak{h}_n$.

\medskip

\noindent
\begin{tabular}{ll} 
Length 1. & $Z(a,0)=V(a)$, $a\ge0$ (here $\mathfrak{r}$ acts trivially). \\[2mm]
Length 2. & $E(a,b)$, with $a+b\equiv m\mod 2$ and 
 $0\le |a-b|\leq m\leq a+b$. \\[2mm]
Length 3. & $Z(a,2)$, $Z(a,2)^*$, $a\ge0$; and \\[1mm]
   & $E_3(c)$ with $c\equiv 2m \mod 4$ and $0\le c< 2m$. \\[2mm]
Length 4. & $Z(a,3)$, $Z(a,3)^*$, $a\ge0$; and \\[1mm]
   & $E_4(t)$, with $t\in\F$ (this exists only if $m\equiv 0\mod 4$). \\[2mm]
Length $\ell\geq 5$. & $Z(a,\ell-1)$, $Z(a,\ell-1)^*$, $a\ge0$.  \\[2mm]
\end{tabular}

\end{theorem}

\subsection{Faithful 
\texorpdfstring{$\big(\sl(2)\ltimes \mathfrak{h}_n\big)$}{}-modules}\label{subsec.Faithful modules}
The faithful uniserial $\big(\sl(2)\ltimes \mathfrak{h}_n\big)$-modules 
were classified, up to isomorphism, in \cite[Theorems 3.5 and 5.2]{CGS2}.
It turns out that there are no faithful uniserial 
$\big(\sl(2)\ltimes \mathfrak{h}_n\big)$-modules of composition length different from $3$.
 Moreover, if 
\[
V=V(a_0)\oplus V(a_1)\oplus V(a_2)
\]
 is socle decomposition (see Definition \ref{rmk:order}) of a faithful uniserial 
 $\big(\sl(2)\ltimes \mathfrak{h}_n\big)$-module, then $a_0=a_2$ and
 an explicit representative of each class can be 
 obtained by conveniently combining the modules $E(a,b)$ for some specific values of $a$ and $b$ as we explain below.
 
Let us start with the $\sl(2)$-module
 $V=V(a_0)\oplus V(a_1)\oplus V(a_2)$ 
 with $a_2=a_0$ such that 
$\frac{m}2,\frac{a_0}2,\frac{a_1}2$ satisfy the triangle condition. 
We now indicate how to obtain an action of $\mathfrak{h}_n=\mathfrak{h}_n(m)\oplus \F z$
on $V$ so that $V$ becomes a faithful uniserial $\big(\sl(2)\ltimes \mathfrak{h}_n\big)$-module. 
Although we know that $a_2=a_0$ we keep the notation $a_2$ 
because we need to indicate that $V(a_0)$ is the socle of $V$ and  $V(a_2)$ corresponds to the third socle factor of $V$. 

First, let $\mathfrak{h}_n(m)$ act on $V$ as follows
\[
0\longleftarrow
V(a_0)\longleftarrow 
V(a_1)\longleftarrow
V(a_2)
\]
where the actions $V(a_1)\rightarrow V(a_0)$ and $V(a_2)\rightarrow V(a_1)$
 are given by 
 \eqref{eq.actionV(m)} (with $a=a_0$, $b=a_1$ and $a=a_1$, $b=a_2$ respectively).
This action of $\mathfrak{h}_n(m)$ on $V$ can be extended to 
$\mathfrak{h}_n$ only in the following cases.
In all of them, $a_0=a_2$ and $z$ acts as an $\sl(2)$-isomorphism $V(a_2)\to V(a_0)$.
\begin{enumerate}
\item[(i)] For $n=1$ (that is $m=1$), $(a_0,a_1,a_2)$ must be
\[
(a_0,a_0+1,a_0),\quad a_0\ge0;\qquad 
(a_0,a_0-1,a_0),\quad a_0\ge1. 
\]
Let us call, respectively, $FU_{a_0}^+$ and $FU_{a_0}^-$ 
the first and second $\big(\sl(2)\ltimes \mathfrak{h}_n\big)$-modules above. 

\medskip

\item[(ii)] For $n=2$ (that is $m=3$), 
$(a_0,a_1,a_2)$ must be 
\[
(0,3,0),\;
(1,4,1),\;
(1,2,1),\;
(4,3,4).
\]
We call these modules $FU_{(0,3,0)}$, 
$FU_{(1,4,1)}$, 
$FU_{(1,2,1)}$ and 
$FU_{(4,3,4)}$ respectively. 

\medskip

\item[(iii)] If $n\geq 3$ (that is $m\ge 5$), $(a_0,a_1,a_2)$ must be
\[
(0,m,0),\;
(1,m+1,1),\;
(1,m-1,1).
\]
We call these modules 
$FU_{(0,m,0)}$, 
$FU_{(1,m+1,1)}$ and 
$FU_{(1,m-1,1)}$ respectively. 
\end{enumerate}

In these modules, the action 
 of the center $\F z$ is given by
 \begin{align}\label{eq.action.V(0)}
z\,v_j^{a_2}=\begin{cases}
-\dfrac{2\sqrt{m+1}}{a+1} v_j^{a_0}, & 
\text{if 
$V=\left\{\begin{array}{l} 
FU_a^+\text{ and } m=1,\\
FU_{(0,3,0)},\; FU_{(1,4,1)}\text{ and } m=3, \\
FU_{(0,m,0)},\; FU_{(1,m+1,1)}\text{ and } m\ge 5.
\end{array}\right.$}
\vspace{3mm}\\
\dfrac{2\sqrt{m+1}}{a+1} v_j^{a_0}, & 
\text{if 
$V=\left\{\begin{array}{l} 
FU_a^-\text{ and } m=1,\\
FU_{(1,2,1)}\text{ and } m=3, \\
FU_{(1,m-1,1)}\text{ and } m\ge 5.
\end{array}\right.$}
\vspace{3mm}\\
-\dfrac{4}{5}v_j^{a_0}, & \text{if $V=FU_{(4,3,4)}\text{ and } m=3$.}
\end{cases}
\end{align}

Let us show more explicitly the $\left(\sl(2)\ltimes \mathfrak{h}_{2}\right)$-module
$FU_{(4,3,4)}$ (here $n=2$ and thus $m=3$). 
Recall that $\{e,h,f,e_s:s=0\dots 3, z\}$ is a basis of 
$\sl(2)\ltimes \mathfrak{h}_{2}$.
The letters $e,h,f,e_s,z$ in the matrices below  
depict the 
non-zero entries of the corresponding basis element. 
The action of $\sl(2)\ltimes \mathfrak{h}_{2}$ on this module is represented 
by a block matrix
\[
\begin{pmatrix}
X_4 & H_1 & Z \\
 0  & X_3 & H_2 \\
 0  &  0  & X_4  
\end{pmatrix}
\]
where 
{\small
\[
X_4= 
\left( 
\begin {array}{ccccc} 
4\,h & 2\,e & 0 & 0 & 0\\ 
2\,f & 2\,h &\sqrt{6}\,e & 0 & 0 \\
0  & \sqrt{6}\,f & 0 & \sqrt{6}\,e & 0 \\
0 & 0 & \sqrt{6}\,f & -2\,h & 2\,e \\
0 & 0 & 0  & 2\,f & -4\,h
\end{array}\right),
\quad 
X_3= 
\left( 
\begin {array}{cccc} 
3\,h & \sqrt{3}\,e & 0 & 0 \\ 
\sqrt{3}\,f & h &2\,e & 0  \\
 0 & 2\,f & -h & \sqrt{3}\,e \\
 0 & 0  & \sqrt{3}\,f & -3\,h
\end{array}\right),  
\]
\begin{align*}
H_1&= 
\left( 
\begin{array}{cccc}
\frac15\,\sqrt{10}\,{e_1} & -\frac15\,\sqrt{10}\,{e_0} & 0 & 0  \\[2mm] 
  \frac15\,\sqrt{10}\,{e_2} & 0 & -\frac15\,\sqrt{10}\,{e_0}& 0  \\[2mm] 
  \frac15\,\sqrt{5}\,{e_3} & \frac15\,\sqrt{5}\,{e_2} & -\frac15\,\sqrt{5}\,{e_1} & -\frac15\,\sqrt{5}\,{e_0} \\[2mm]
0 & \frac15\,\sqrt{10}\,{e_3} & 0 & -\frac15\,\sqrt{10}\,{e_1} \\[2mm]
0 & 0 & \frac15\,\sqrt{10}\,{e_3} & -\frac15\,\sqrt{10}\,{e_2} 
\end{array} \right), \\[2mm]
H_2&= 
\left( 
\begin{array}{ccccc}
\frac15\,\sqrt{10}\,{e_2} & -\frac15\,\sqrt{10}\,{e_1} & \frac15\,\sqrt{5}\,{e_0} & 0 & 0\\[2mm] 
\frac15\,\sqrt{10}\,{e_3} & 0 & -\frac15\,\sqrt{5}\,{e_1} & \frac15\,\sqrt{10}\,{e_0} & 0 \\[2mm]
0  & \frac15\,\sqrt{10}\,{e_3} & -\frac15\,\sqrt{5}\,{e_2} & 0 & \frac15\,\sqrt{10}\,{e_0} \\[2mm]
 0  & 0 & \frac15\,\sqrt{5}\,{e_3} & -\frac15\,\sqrt{10}\,{e_2} & \frac15\,\sqrt{10}\,{e_1}
\end{array} \right),
\end{align*} }

\noindent
and $Z=-\frac{4}{5}\,z\,I_{5}$ (here $I_5$ is the identity matrix of size $5$). 
This module is exceptionally rare: it is remarkable that the bracket of any matrix
of type $H_1$ with a matrix of type $H_2$ produces a multiple of the identity. 
This is related to some non-trivial zeros of the $6j$-symbols 
(see \cite{CS_JofAlg}) and exceptional zeros of the Clebsch-Gordan coefficientes (see for instance the entries (2,2) and (4,3) of $H_1$).

The following theorem summarizes this information and was proved in \cite{CGS2}.

\begin{theorem}\label{thm.CS2_Classification}
The following list describes all the isomorphism classes of faithful uniserial representations 
of $\sl(2)\ltimes \mathfrak{h}_n$.
\begin{enumerate}
\item[(i)] For $n=1$: $FU_a^+$, $a\ge0$, and $FU_a^-$, $a\ge1$.
\item[(ii)] For $n=2$: $FU_{(0,3,0)}$, 
$FU_{(1,4,1)}$, 
$FU_{(1,2,1)}$ and 
$FU_{(4,3,4)}$.
\item[(iii)] For $n\geq 3$: 
$FU_{(0,m,0)}$, 
$FU_{(1,m+1,1)}$ and 
$FU_{(1,m-1,1)}$ ($m=2n-1$).
\end{enumerate}
Each of these modules is isomorphic to its own dual.
\end{theorem}

We close this section pointing out that
all the faithful uniserial $\big(\sl(2)\ltimes \mathfrak{h}_n\big)$-modules, except $FU_{(4,3,4)}$, belong, in some sense, to the same kind of modules (we will say something more about this after Conjecture \ref{conj:length2}). 
Therefore, we introduce the following definition.
\begin{definition}
All faithful uniserial $\big(\sl(2)\ltimes \mathfrak{h}_n\big)$-modules that are not isomorphic to $FU_{(4,3,4)}$ will be referred to as 
\emph{standard faithful uniserials}.
\end{definition}

\section{The tensor product of two uniserial 
\texorpdfstring{$\big(\sl(2)\ltimes \mathfrak{h}_n\big)$-modules}{}}
\label{sec.4}

In \cite[Theorem 3.5]{CGR} we obtained the $\sl(2)$-module structure of the socle of the tensor product of two uniserial $\big(\sl(2)\ltimes \mathfrak{a}_m\big)$-modules of type $Z$. 
Therefore (see \S\ref{sec.3}), we already know the $\sl(2)$-module structure of the socle of the tensor product $V_1\otimes V_2$ of two uniserial $\big(\sl(2)\ltimes \mathfrak{h}_n\big)$-modules in the cases where $V_1$ and $V_2$ are non-faithful of type $Z$.
We summarize this in Theorem \ref{thm:main} below.

In this section we obtain 
the $\sl(2)$-module structure of the socle of the tensor product 
$V_1\otimes V_2$ when both $V_1$ and $V_2$ are standard faithful uniserial modules, 
or when one of them is standard faithful and the other one is uniserial of type $Z$.
From this, we derive a complete description of the space of 
intertwining operators 
$\text{Hom}_{\sl(2)\ltimes \mathfrak{h}_n}(V_1, V_2)$ in all
these cases.

\subsection{The non-faithful case and the crucial conjecture}
Let us recall some of the results obtained in \cite{CGR}, mainly  
Conjecture 3.4 and Theorem 3.5 that describes 
the $\sl(2)$-module structure of the socle of the tensor product of two uniserial $\big(\sl(2)\ltimes \mathfrak{a}_m\big)$-modules of type $Z$. 

If $U$ is a $\big(\sl(2)\ltimes \mathfrak{h}_n\big)$-module, 
since $\mathfrak{h}_n=[\sl(2),\mathfrak{h}_n]$,
it follows from Lemma \ref{lemma:soc} that 
\begin{equation}\label{eq.soc=invariants}
\text{soc}(U)=U^{\mathfrak{h}_n}.
\end{equation}
Therefore, if
$V=V(a_0)\oplus \hdots \oplus V(a_{\ell})$ and 
$W=V(b_0)\oplus \hdots \oplus V(b_{\ell'})$
are the socle decomposition of two
$\big(\sl(2)\ltimes \mathfrak{h}_n\big)$-modules,
then
\begin{align}
\text{soc}(V\otimes W)&= \bigoplus_{t=0}^{\ell+\ell'}\left(\text{soc}(V\otimes W)\cap \bigoplus_{i+j=t} V(a_i)\otimes V(b_j) \right)\notag \\
&= \bigoplus_{t=0}^{\ell+\ell'}\left(\bigoplus_{i+j=t}V(a_i)\otimes V(b_j) \right)^{\mathfrak{h}_n}.
\end{align}
For $t=0,\hdots,\ell+\ell'$, we define
\begin{equation*}
S_t=S_t(V,W)=\left(\bigoplus_{i+j=t}V(a_i)\otimes V(b_j) \right)^{\mathfrak{h}_n}.
\end{equation*}
Hence, as $\sl(2)$-modules,
\begin{equation*}
\text{soc}(V\otimes W)
=\bigoplus_{t=0}^{\ell+\ell'}S_t 
\end{equation*}
and 
\[
S_0=\text{soc}(V)\otimes \text{soc}(W)=V(a_0)\otimes V(b_0).
\]

The following theorem is the same as \cite[Theorem 3.5]{CGR} but stated in terms of non-faithful uniserial 
$\big(\sl(2)\ltimes \mathfrak{h}_n\big)$-modules.

\begin{theorem}\label{thm:main}
Let $V_1=V(a_0)\oplus \hdots \oplus V(a_\ell)$ and $V_2=V(b_0)\oplus \hdots \oplus V(b_{\ell'})$ be socle decomposition of two non-faithful uniserial $\big(\sl(2)\ltimes \mathfrak{h}_n\big)$-modules of type $Z$. Then,
$S_t=0$ for all $t>\min\{\ell,\ell'\}$,
\[
S_0\simeq \bigoplus_{k=0}^{\min\{a_0,b_0\}}V(a_0+b_0-2k),
\]
and, for $t=1,\dots,\min\{\ell,\ell'\}$, we have
\begin{enumerate}
\item[(i)] If $V_1=Z(a_0,\ell)$ and $V_2=Z(b_0,\ell')$, then
\[
S_t \simeq V(a_0+b_0+tm).
\]

\vspace{3mm}

\item[(ii)] If $V_1=Z(a_0,\ell)$ and $V_2=Z(b_{\ell'},\ell')^*$, then 
\[
S_t \simeq 
\begin{cases}
0, & \text{if $tm>b_0-a_0$;} \\
V(b_0-a_0-tm), & \text{if $tm\le b_0-a_0$.}
\end{cases}
\]
\vspace{3mm}

\item[(iii)] If $V_1=Z(a_\ell,\ell)^*$ and $V_2=Z(b_{\ell'},\ell')^*$, then 
$S_t=0$.
\end{enumerate}
\end{theorem}

One of the main steps towards proving the above theorem 
was to prove certain instances of the following conjecture
(see \cite[Conjecture 3.4]{CGR}).

\begin{conj}\label{conj:length2}
 Let 
$V_1=E(a,b)$ and 
$V_2=E(c,d)$ (two uniserial $\sl(2)\ltimes\mathfrak{a}_m$-modules of length 2)
and assume that $a<c$, or
$a=c$ and $b\le d$. 
Then $S_2=0$ in all cases and $S_1=0$ except in the following cases. 
\begin{enumerate}[\hspace{3mm}]
\item[\tiny$\bullet$]
Case 1: $[a,b]=[0,m]$. Here $S_1\simeq V(d)$. 
\item[\tiny$\bullet$] Cases 2: Here $a>0$.
\begin{enumerate}[\hspace{3mm}]
\item[--] Case 2.1: $a+b=c+d=m$ with $d-a=b-c\ge0$. 
Here $S_1\simeq V(d-a)$.
\item[--] Case 2.2: $b-a=d-c=m$. Here $S_1\simeq V(d+a)$.
\item[--] Case 2.3: $b-a=c-d=m$ with $d-a=c-b\ge0$. 
Here $S_1\simeq V(d-a)$.
 \end{enumerate}
\item[\tiny$\bullet$] Case 3: $[c,d]=[b,a]$. Here $S_1\simeq V(0)$.
 \end{enumerate}
 \end{conj}
 
In order to prove 
Theorem \ref{thm:main}, we proved in \cite[Theorem 3.3]{CGR} the cases 2.2 and 2.3 (and certain converse statement). 
Now, in this paper, we need to prove
part of case 1 and case 2.1 of the conjecture (together with certain converse statement)
in order to prove 
Theorems \ref{thm.fielvsnofiel} and \ref{thm.fielvsfiel}.
This is established in Theorem \ref{prop.conjecture} below.
We point out that this theorem leaves out the 
uniserial $\big(\sl(2)\ltimes\mathfrak{a}(3)\big)$-modules $E(3,4)$ and $E(4,3)$, and a consequence of this is that 
Theorems \ref{thm.fielvsnofiel} and \ref{thm.fielvsfiel} are restricted to 
the standard faithful modules, leaving the exceptional faithful 
uniserial 
$\big(\sl(2)\ltimes \mathfrak{h}_2\big)$-module $FU_{(4,3,4)}$ out of our results.

\begin{theorem}\label{prop.conjecture}
Let $V_1=E(a,b)$ with $a+b=m$ and $a,b\neq 0$.
Let $V_2=V(c)\oplus V(d)$ be the socle decomposition of a 
uniserial $\big(\sl(2)\ltimes V(m)\big)$-module. 
Then

\emph{(i)} If $V_2\simeq E(c,d)$ with $c+d=m$ and $0<a\leq c < m$, then, 
as $\sl(2)$-modules,
\[S_1(V_1,V_2)\simeq \begin{cases}
V(d-a), & \text{if $d-a=b-c\geq 0$}\\
0, & \text{otherwise}.
\end{cases}\]

\

\emph{(ii)} If $V_2\simeq Z(c,1)\simeq E(c,c+m)$, then, as $\sl(2)$-modules,
\[S_1(V_1,V_2)\simeq \begin{cases}
V(b), & \text{if $c=0$}\\
0, & \text{if $c\neq 0$}
\end{cases}.\]

\

\emph{(iii)} If $V_2\simeq Z(d,1)^*\simeq E(d+m,d)$, then
$S_1(V_1,V_2)=0$.

\end{theorem}

\

The proof of this result is very technical and it will be given 
in \S\ref{sec.proof_conj}.

\subsection{The faithful case}
We will now focus on the tensor product of two uniserial $\big(\sl(2)\ltimes \mathfrak{h}_n\big)$-modules where one of the factors is a standard faithful uniserial. 

Let $V=V(a_0)\oplus V(a_1)\oplus V(a_2)$ be the socle decomposition of a faithful uniserial $\big(\sl(2)\ltimes \mathfrak{h}_n\big)$-module and let $W=V(b_0)\oplus \hdots \oplus V(b_{\ell})$, with $\ell\geq 1$, be the socle decomposition of a (not necessarily faithful)
uniserial $\big(\sl(2)\ltimes \mathfrak{h}_n\big)$-module. 
By Theorems \ref{thm.CS_Classification} and \ref{thm.CS2_Classification} we have that
\begin{align*}
\mathfrak{h}_n(m)\cdot V(a_i)\subset V(a_{i-1}) &\;\;\text{and}\;\; \mathfrak{h}_n(m)\cdot V(b_j)\subset V(b_{j-1})\oplus V(b_{j-2}) \\
z\cdot V(a_i) \subset V(a_{i-2}) &\;\;\text{and}\;\; z\cdot V(b_j)\subset V(b_{j-2})
\end{align*} 
for all $i=0,1,2$ and $0\leq j\leq \ell$ (for convenience we assume $V(a_i)=V(b_j)=0$ if $i,j<0$). 

Given $v\in V(a_i)\otimes V(b_j)$, let
\begin{equation}\label{action.general}
e_sv=(e_sv)_1+(e_sv)_2+(e_sv)_3 \;\; \text{and}\;\; zv=(zv)_1+(zv)_2
\end{equation}
where
\begin{align*}
(e_sv)_1&\in V(a_{i-1})\otimes V(b_j),& (zv)_1&\in V(a_{i-2})\otimes V(b_j),\\
(e_sv)_2&\in V(a_i)\otimes V(b_{j-1}),& (zv)_2&\in V(a_i)\otimes V(b_{j-2}),\\
(e_sv)_3&\in V(a_i)\otimes V(b_{j-2}).&&
\end{align*}
Note that $(e_sv)_3= 0$ if $W$ is not isomorphic to $E_4$ and that $(zv)_2=0$ if $W$ is not a faithful uniserial module.

\begin{lemma}\label{lemma:action neq}
Let $V_1=V(a_0)\oplus V(a_1)\oplus V(a_2)$ and $V_2=V(b_0)\oplus \hdots \oplus V(b_{\ell})$, with $\ell\geq 1$, be the socle decomposition of two uniserial $\big(\sl(2)\ltimes \mathfrak{h}_n\big)$-modules, where $V_1$ is faithful (not necessarily standard). 
If $v_0\in V(a_{i_0})\otimes V(b_{j_0})$ is a highest weight vector, then:
\begin{enumerate}
\item[(i)] $(e_sv_0)_1=0$ for all $s=0,\hdots, m$ if and only if $i_0=0$.
\item[(ii)] $(e_sv_0)_2=0$ for all $s=0,\hdots, m$ if and only if $j_0=0$.
\item[(iii)] $(zv_0)_1=0$ if and only if $i_0\neq 2$.
\item[(iv)] If $V_2$ is faithful, then $(zv_0)_2=0$ if and only if $j_0\neq 2$.
\end{enumerate}
\end{lemma}
\begin{proof} 
Since the action of $\mathfrak{h}_n(m)$ on any uniserial $\big(\sl(2)\ltimes \mathfrak{h}_n\big)$-module  is the same as the action of $\mathfrak{a}_{m}$ in the corresponding $\big(\sl(2)\ltimes \mathfrak{a}_{m}\big)$-module, 
 cases (i) and (ii) are immediate consequences of \cite[Lemma 3.1]{CGR}.

By symmetry, it is sufficient to prove (iii) to obtain (iv), so let us prove (iii). 
If $c$ is the weight of $v_0$, 
we can assume that $v_0=v_0^{a_{i_0},b_{j_0},c}$.
It follows from 
\eqref{eq.Vc_en_tensor} and \eqref{action.general} that
\begin{align}
(zv_0)_1
& =(zv_0^{a_{i_0},b_{j_0},c})_1 \notag \\ \label{eq.lem1}
& =\sum_{i+j=\tfrac{a_{i_0}+b_{j_0}-c}{2}}CG(\tfrac{a_{i_0}}{2},\,\tfrac{a_{i_0}}{2}-i;\,\tfrac{b_{j_0}}{2},\;\tfrac{b_{j_0}}{2}-j\,|\,\tfrac{c}{2},\tfrac{c}{2})\,zv_i^{a_{i_0}}\otimes v_j^{b_{j_0}}.
\end{align}
From the definition of the modules $FU^{\pm}_a$ for $n=1$; $FU_{(0,3,0)}$,$FU_{(1,4,1)}$, $FU_{(1,2,1)}$ and $FU_{(4,3,4)}$ for $n=2$; and the modules $FU_{(0,m,0)}$, $FU_{(1,m+1,1)}$ and $FU_{1,m-1,1}$ for $n\ge 3$ (here $m=2n-1$), we know that $zv_i^{a_{i_0}}=0$ if $i_0\neq 2$.
 Therefore, if $i_0\neq 2$ then 
\[
(zv_0)_1=0.
\]
On the other hand, if $i_0=2$, we know from \eqref{eq.action.V(0)} that $zv_i^{a_2}=\lambda v_i^{a_0}$, where $\lambda$ is a non-zero scalar independent of $i$, $0\leq i\leq a_2$. 
Thus, the equation \eqref{eq.lem1} becomes 
\begin{align*}
(zv_0)_1=\lambda\,\sum_{i+j=\tfrac{a_{2}+b_{j_0}-c}{2}}CG(\tfrac{a_{2}}{2},\,\tfrac{a_{2}}{2}-i;\,\tfrac{b_{j_0}}{2},\;\tfrac{b_{j_0}}{2}-j\,|\,\tfrac{c}{2},\tfrac{c}{2})\,v_i^{a_{0}}\otimes v_j^{b_{j_0}}.
\end{align*}
In this sum, the term corresponding to $i=0$, has a non-zero Clebsch-Gordan coefficient, indeed 
\[
CG(\tfrac{a_{2}}{2},\,\tfrac{a_{2}}{2};\,\tfrac{b_{j_0}}{2},\;\tfrac{c-a_2}{2}\,|\,\tfrac{c}{2},\tfrac{c}{2})=
\sqrt{\tfrac{(c+1)!\,a_{2}!}{\left(\tfrac{a_{2}+b_{j_0}+c}{2}+1 \right)!\,\left(\tfrac{a_{2}+c-b_{j_0}}{2} \right)!}}\neq 0.
\]
Since all terms are linearly independent, we obtain $(zv_0)_1\neq 0$.  
\end{proof}

\begin{proposition}\label{prop:Zocalo_t}
Let $V_1=V(a_0)\oplus V(a_1)\oplus V(a_2)$ and $V_2=V(b_0)\oplus \hdots \oplus V(b_{\ell})$, with $\ell\geq 1$, 
be the socle decomposition of two uniserial $\big(\sl(2)\ltimes \mathfrak{h}_n\big)$-modules, where $V_1$ is standard faithful. 
Then, $S_0=V(a_0)\otimes V(b_0)$ and 
\begin{enumerate}
\item[(i)] 
$S_t=0$ for all $t> \min\{2,\ell\}$.
If $S_t\ne 0$, $t=1,2$, then it is irreducible as 
$\sl(2)$-module
and if $v$ is a non-zero highest weight vector in $S_t$ 
of weight $\mu$, then $v=\sum_{i=0}^tv_i$ with 
$v_i$ a non-zero highest weight vector in 
$V(a_i)\otimes V(b_{t-i})$, of weight $\mu$, for all $i=0,\dots,t$.
\item[(ii)] $S_2=0$ if $V_2$ is non-faithful.
\item[(iii)] $S_1(V_1,V_2)\simeq 
S_1\big(E(a_0,a_1), E(b_0,b_1)\big)$.
\item[(iv)] If $V_2$ is also standard faithful, then 
$S_2\ne 0$ if and only if $V_1\simeq V_2$ (that is 
$a_i=b_i$, $i=0,1,2$) and in this case 
$S_2\simeq V(0)$.
\end{enumerate}
\end{proposition}

\begin{proof}
The proof of this proposition is very similar to that of 
Proposition 3.2 in \cite{CGR}.
We fix $t> 0$ and we assume that there is a non-zero
highest weight vector $u$ of weight $\mu$, 
\[
u=\sum_{i+j=t} u_{i,j}
\in \Big(\bigoplus_{i+j=t} V(a_{i})\otimes V(b_j)\Big)^{\mathfrak{h}_n}\ne 0, \qquad u_{i,j} \in V(a_{i})\otimes V(b_j).
\]
Since $V(a_{i})\otimes V(b_j)$ is an $\mathfrak{sl}(2)$-submodule, it follows that $u_{i,j}$ is either zero or a highest weight vector of weight $\mu$. 
Let
\[
 I_t^\mu=\{(i,j): 0\le i\le 2,\; 0\le j \le\ell,\;i+j=t \text{ and } u_{i,j}\ne 0\}.
\]
Since $u\ne0$, it follows that 
$I_t^\mu\ne\emptyset$ and 
\[
 u=\sum_{(i,j)\in I_t^\mu}q_{i,j}\,v_0^{a_i,b_j,\mu}
\]
for certain non-zero scalars $0\ne q_{i,j}\in\F$. 
Now, it follows from items (i) and (ii) in Lemma \ref{lemma:action neq} (see the details in 
 \cite{CGR}[Proposition 3.2]) 
that 
\begin{equation}\label{eq:It}
I_t^\mu=\{(0,t),(1,t-1),\dots,(t,0)\}.
\end{equation}
Now, again, items (i) and (ii) in Lemma \ref{lemma:action neq} imply 
that such a non-zero $u$ cannot exist if $t> \min\{2,\ell\}$ and thus $S_t=0$. This proves (i).
Furthermore, (ii) follows similarly by applying
item (iii) in Lemma \ref{lemma:action neq}. 

(iii) is clear from the definition of $S_1$.

Let us prove (iv).
Assume that $V_2=V(b_0)\oplus V(b_1)\oplus V(b_2)$ is standard faithful, 
and suppose that 
\begin{equation}\label{eq.u}
 u=q_{0,2}\,v_0^{a_0,b_2,\mu}
 +q_{1,1}\,v_0^{a_1,b_1,\mu}
 +q_{2,0}\,v_0^{a_2,b_0,\mu}\ne 0
\end{equation}
is a highest weight vector, of weight $\mu$, in $S_2$. 
We already know that $q_{0,2}, q_{1,1},q_{2,0}\ne0$ and, moreover,
we must have
\[
q_{0,2}\,v_0^{a_0,b_2,\mu}
 +q_{1,1}\,v_0^{a_1,b_1,\mu}\in 
S_1\big(E(a_0,a_1),E(b_1,b_2))\ne0
\]
and 
\[
 q_{1,1}\,v_0^{a_1,b_1,\mu}
 +q_{2,0}\,v_0^{a_2,b_0,\mu} \in 
S_1\big(E(a_1,a_2), E(b_0,b_1))\ne0.
\]

Theorems \ref{thm:main} and
\ref{prop.conjecture} imply that 
it is impossible to have 
\[
S_1\big(E(a_0,a_1), E(b_1,b_2)\big)\ne0\quad\text{ and } \quad 
S_1\big(E(a_1,a_2), E(b_0,b_1)\big)\ne0
\]
unless $(a_0,a_1,a_2)=(b_0,b_1,b_2)$.
This follows by considering all the cases with 
$(a_0,a_1,a_2)$ and $(b_0,b_1,b_2)$ running over (see \S\ref{subsec.Faithful modules})
\begin{enumerate}
\item[(i)] if $n=1$ (that is $m=1$)
\[
(k_0,k_0+1,k_0),\quad k_0\ge0;\qquad 
(k_0,k_0-1,k_0),\quad k_0\ge1; 
\]
\item[(ii)] if $n\geq 2$ (that is $m\ge 3$), 
\[
(0,m,0),\;
(1,m+1,1),\;
(1,m-1,1);
\]
\end{enumerate}
(it saves time noticing 
 that 
 $E(a_0,a_1)^*\simeq E(a_1,a_2)$ and $E(b_0,b_1)^*\simeq E(b_1,b_2)$).

Finally, let $(a_0,a_1,a_2)=(b_0,b_1,b_2)$.
We know, from Theorems \ref{thm:main} and
\ref{prop.conjecture}, that 
\[
S_1\big(E(a_0,a_1), E(b_1,b_2)\big)\simeq 
S_1\big(E(a_1,a_2), E(b_0,b_1)\big)
\simeq V(0).
\]
This implies that there is, up to a scalar, a unique element $u$ as in \eqref{eq.u}
such that $\mathfrak{h}_n(m)u=0$. This implies that 
$zu=0$ and hence $u\in S_2$. 
This completes the proof. 
\end{proof}

 We are now in a position to prove the main theorems of the paper.

\begin{theorem}\label{thm.fielvsnofiel}
Let $V_1=V(a_0)\oplus V(a_1)\oplus V(a_2)$ and $V_2=V(b_0)\oplus \hdots \oplus V(b_\ell)$ be the socle decomposition of two uniserial $\big(\sl(2)\ltimes \mathfrak{h}_n\big)$-modules, where $V_1$ is standard faithful and $V_2$ is of type $Z$. Then  
\[\text{soc}(V_1\otimes V_2)= S_0\oplus S_1\]
where, as $\sl(2)$-modules,
\[S_0=\text{soc}(V_1)\otimes \text{soc}(V_2)\simeq \bigoplus_{k=0}^{\min\{a_0,b_0\}}V(a_0+b_0-2k)\]
and the following tables describe $S_1$ as $\sl(2)$-modules (recall that $m=2n-1$):

\medskip

\noindent
Case $n=1$ ($m=1$). 

\noindent
\begin{tabular}{|c|c|c|}
\hline
${}_{\displaystyle V_1}\backslash 
{\displaystyle V_2}$ & 
$\begin{array}{l}
Z(b_0,\ell) 
\end{array}$  & 
$\begin{array}{l}
Z(b_\ell,\ell)^* 
\end{array}$
\rule[-3mm]{0mm}{8mm}\\
\hline
$\begin{array}{l}
FU_{a_0}^+
\end{array}$ & 
$V(a_0+b_0+m).$ & 
$\begin{array}{rl}
V(b_1-a_0), & \text{if $b_1\ge a_0$, $\ell \ge 1$;}\\[2mm]
0, & \text{otherwise.}
\end{array}$
\rule[-6mm]{0mm}{14mm}\\
\hline
$\begin{array}{l}
FU_{a_0}^- \\[1mm]
\text{\small $a_0\ge 1$}
\end{array}$ & 
$\begin{array}{rl}
V(a_0-b_1), & \text{if $a_0\ge b_1$, $\ell \ge 1$;}\\[2mm]
0,& \text{otherwise.}
 \end{array}$ & 
$0.$
\rule[-6mm]{0mm}{14mm}\\
\hline
\end{tabular}

\bigskip

\noindent
Case $n>1$ ($m\ge3$). 

\noindent
\begin{tabular}{|c|c|c|}
\hline
${}_{\displaystyle V_1}\backslash 
{\displaystyle V_2}$ & 
$\begin{array}{l}
Z(b_0,\ell) 
\end{array}$  & 
$\begin{array}{l}
Z(b_\ell,\ell)^* 
\end{array}$
\rule[-3mm]{0mm}{8mm}\\
\hline
$\begin{array}{l}
FU_{(a_0,a_0+m,a_0)} \\[1mm]
\text{\small $a_0=0,1$}
\end{array}$ & 
$V(a_0+b_0+m).$ & 
$\begin{array}{rl}
V(b_1-a_0), & \text{ if $a_0\leq b_1$;}\\[2mm]
0, & \text{ if $a_0>b_1$.}
\end{array}$
\rule[-6mm]{0mm}{14mm}\\
\hline
$\begin{array}{l}
FU_{(1,m-1,1)}
\end{array}$ & 
$\begin{array}{rl}
V(m-1), & \text{if $b_0=0$;}\\[2mm]
0,& \text{otherwise.}
 \end{array}$ & 
$0.$
\rule[-6mm]{0mm}{14mm}\\
\hline
\end{tabular}
\end{theorem}

\begin{proof}
We know from Proposition \ref{prop:Zocalo_t} that 
\[\text{soc}(V_1\otimes V_2)\simeq \text{soc}(V_1)\otimes \text{soc}(V_2)\oplus S_1\]
(in particular we know that $S_t(V_1,V_2)=0$ for all $t\ge2$).
The decomposition of $\text{soc}(V_1)\otimes \text{soc}(V_2)$
follows from the Clebsch-Gordan formula. 
We now describe $S_1$ in each case. 

Let us consider the submodules 
\begin{align*}
U_1& =V(a_0)\oplus V(a_1)\subset V_1, \text{ and} \\
U_2& =V(b_0)\oplus V(b_1)\subset V_2. 
\end{align*}
We know that $U_1$ and $U_2$ are 
non-faithful uniserial $\big(\sl(2)\ltimes \mathfrak{h}_n\big)$-modules. 
From Proposition \ref{prop:Zocalo_t} item (iii) we know that 
\begin{align*}\label{def.U's}
S_1(V_1,V_2)&\simeq S_1(U_1,U_2),
\end{align*}

If $V_1$ is $FU_{a_0}^+$ or $FU_{(0,m,0)}$ or $FU_{(1,1+m,1)}$ then 
$U_1\simeq Z(a_0,1)$ and $S_1$ is obtained from Theorem \ref{thm:main}. 
If $V_1=FU_{a_0}^-$, then $U_1=Z(a_0-1,1)^*$ and $S_1$ is also obtained from Theorem \ref{thm:main}. 

Finally, if $V_1=FU_{(1,m-1,1)}$, then $U_1=E(1,m-1)$ and Theorem \ref{prop.conjecture} implies the remaining cases. 
\end{proof}

\begin{theorem}\label{thm.fielvsfiel}
Let $V_1=V(a_0)\oplus V(a_1)\oplus V(a_2)$ and $V_2=V(b_0)\oplus V(b_1) \oplus V(b_2)$ be the socle decomposition of two standard faithful uniserial $\big(\sl(2)\ltimes \mathfrak{h}_n\big)$-modules. Then
\[\text{soc}(V_1\otimes V_2)= S_0\oplus S_1\oplus S_2\]
where 
\[S_0=\text{soc}(V_1)\otimes \text{soc}(V_2)\simeq \bigoplus_{k=0}^{\min\{a_0,b_0\}}V(a_0+b_0-2k)\] 
and the following tables describe $S_1$ and $S_2$  as $\sl(2)$-modules 
($m=2n-1$):

\bigskip

\noindent Case $n=1$ ($m=1$), structure of $S_1$.

\noindent
\begin{tabular}{|c|c|c|}
\hline
${}_{\displaystyle V_1}\backslash 
{\displaystyle V_2}$ & 
$\begin{array}{l}
{ FU^+_{b_0}}
\end{array}$  & 
$\begin{array}{l}
{ FU^-_{b_0}}\\
\text{\small $b_0\ge 1$}
\end{array}$ 
\rule[-3mm]{0mm}{10mm}\\
\hline
$\begin{array}{l}
{FU^+_{a_0}}
\end{array}$ & 
$V(a_0+b_0+1). $ & 
$\begin{array}{rl}
V(b_0-a_1),& \text{if $a_1\le b_0$;} \\[1mm]
0,& \text{otherwise.} 
\end{array}$ 
\rule[-6mm]{0mm}{16mm}\\
\hline
$\begin{array}{l}
{FU^-_{a_0}}\\
\text{\small $a_0\ge 1$}
\end{array}$ & 
$\begin{array}{rl}
V(a_0-b_1),& \text{if $b_1\le a_0$;} \\[1mm]
0,& \text{otherwise.} 
\end{array}$ & 
$0.$
\rule[-6mm]{0mm}{16mm}\\
\hline
\end{tabular}

\bigskip

\noindent Case $n=1$ ($m=1$), structure of $S_2$.

\noindent
\begin{tabular}{|c|c|c|}
\hline
${}_{\displaystyle V_1}\backslash 
{\displaystyle V_2}$ & 
$\begin{array}{l}
{FU^+_{b_0}} 
\end{array}$  & 
$\begin{array}{l}
{FU^-_{b_0}} \\
\text{\small $b_0\ge 1$}
\end{array}$ 
\rule[-3mm]{0mm}{10mm}\\
\hline
$\begin{array}{l}
{FU^+_{a_0}}
\end{array}$ & 
$\begin{array}{rl}
V(0),& \text{if $a_0= b_0$;}\\[1mm]
0, & \text{if $a_0\neq b_0$.} 
\end{array}$& 
$0.$ 
 \rule[-6mm]{0mm}{14mm}\\
\hline
$\begin{array}{l}
{FU^-_{a_0}}\\
\text{\small $a_0\ge 1$}
\end{array}$ & 
$0.$& 
$\begin{array}{rl}
V(0),& \text{if $a_0=b_0$;} \\[1mm] 
0,& \text{if $a_0\neq b_0$.} 
\end{array}$ 
\rule[-6mm]{0mm}{14mm}\\
\hline
\end{tabular}

\bigskip

\noindent Case $n>1$ ($m\ge 3$), structure of $S_1$.

\noindent
\begin{tabular}{|c|c|c|}
\hline
${}_{\displaystyle V_1}\backslash 
{\displaystyle V_2}$ & 
$\begin{array}{l}
FU_{(b_0,b_0+m,b_0)} \\
\text{\small $b_0=0,1$}
\end{array}$  & 
$\begin{array}{l}
FU_{(1,m-1,1)}
\end{array}$ 
\rule[-3mm]{0mm}{9mm}\\
\hline
$\begin{array}{l}
FU_{(a_0,a_0+m,a_0)}\\
\text{\small $a_0=0,1$}
\end{array}$ & 
$V(a_0+b_0+m). $ & 
$\begin{array}{rl}
V(m-1), & \text{if $a_0=0$;}\\[1mm]
0,& \text{otherwise.} 
\end{array}$ 
\rule[-6mm]{0mm}{16mm}\\
\hline
$\begin{array}{l}
FU_{(1,m-1,1)}
\end{array}$ & 
$\begin{array}{rl}
V(m-1), & \text{if $b_0=0$;}\\[1mm]
0,& \text{otherwise.} 
\end{array}$ & 
$V(m-2).$
\rule[-6mm]{0mm}{16mm}\\
\hline
\end{tabular}

\bigskip

\noindent Case $n>1$ ($m\ge 3$), structure of $S_2$.

\noindent
\begin{tabular}{|c|c|c|}
\hline
${}_{\displaystyle V_1}\backslash 
{\displaystyle V_2}$ & 
$\begin{array}{l}
FU_{(b_0,b_0+m,b_0)}\\
\text{\small $b_0=0,1$} 
\end{array}$  & 
$\begin{array}{l}
FU_{(1,m-1,1)} 
\end{array}$ 
\rule[-3mm]{0mm}{9mm}\\
\hline
$\begin{array}{l}
FU_{(a_0,a_0+m,a_0)}\\
\text{\small $a_0=0,1$}
\end{array}$ & 
$\begin{array}{rl}
V(0),& \text{if $a_0= b_0$;}\\[1mm]
0, & \text{if $a_0\neq b_0$.} 
\end{array}$ & 
$0.$ 
 \rule[-6mm]{0mm}{14mm}\\
\hline
$\begin{array}{l}
FU_{(1,m-1,1)}
\end{array}$ & 
$0.$& 
$V(0).$ 
\rule[-6mm]{0mm}{14mm}\\
\hline
\end{tabular}
\end{theorem}

\

\begin{proof}
As in the previous theorem, we know from Proposition \ref{prop:Zocalo_t} that 
\[\text{soc}(V_1\otimes V_2)\simeq \text{soc}(V_1)\otimes \text{soc}(V_2)\oplus S_1\oplus S_2\]
and the decomposition of $\text{soc}(V_1)\otimes \text{soc}(V_2)$
follows from the Clebsch-Gordan formula. 
We now describe $S_1$ and $S_2$ in each case. 

First, we consider $S_1$. 
We know from Proposition \ref{prop:Zocalo_t} 
that 
\[
S_1(V_1,V_2)\simeq 
S_1\big(E(a_0,a_1), E(b_0,b_1)\big).
\]
If $V_1=FU_{a_0}^+$ and $V_2=FU_{b_0}^+$, we know from Theorem \ref{thm:main} that
\[
S_1(V_1,V_2)=S_1(Z(a_0,1),Z(b_0,1))=V(a_0+b_0+m).
\]
If $V_1=FU_{a_0}^+$ and $V_2=FU_{b_0}^-$, we have
\[
S_1(V_1,V_2)=\begin{cases}
S_1(Z(a_0,1),Z(b_1,1)^*),& \text{if $m=1$;}\\[1mm]
S_1(Z(a_0,1),E(1,1-m)), & \text{if $m>1$ and $b_0=1$;}
\end{cases}
\]
and it follows from Theorems \ref{thm:main} and \ref{prop.conjecture} that
\[
S_1(V_1,V_2)=\begin{cases}
V(b_1-a_0) & \text{if $m=1$;}\\
V(m-1), & \text{if $m>1$, $a_0=0$ and $b_0=1$;}\\
0,& \text{otherwise.}
\end{cases}
\]
This completes the description of $S_1$.

The result for $S_2$ follows from item (iv) in Proposition \ref{prop:Zocalo_t}.
\end{proof}

\section{Intertwining operators}
\label{sec.5}

In this section we obtain, from 
Theorems \ref{thm.fielvsnofiel} and \ref{thm.fielvsfiel},
the space of intertwining operators 
between the uniserial representations of 
$\g=\sl(2)\ltimes \mathfrak{h}_n$ 
considered in the previous section.

Recall that, for any pairs of $\g$-modules
$U_1$ and $U_2$, we know that 
\[
\text{Hom}_\g(U_1,U_2)\simeq (U_1^*\otimes U_2)^\g= \big((U_1^*\otimes U_2)^{\mathfrak{h}_n}\big)^{\sl(2)}.
\]
It follows that $\text{Hom}_{\g}(U_1,U_2)$ is isomorphic to 
$\text{soc}(U_1^*\otimes U_2)^{\sl(2)}$ (see \eqref{eq.soc=invariants}).
Thus, we must identify the cases in which $S_t(U_1^*,U_2)^{\sl(2)}\ne0$ for $t=0,1,2$.

So, for the rest of the section, let us assume that 
\begin{align}\label{eq.VV}
V& = V(a_0)\oplus V(a_1)\oplus V(a_2), \quad\text{($a_0=a_2$)} \\ \label{eq.WW}
W&=V(b_0)\oplus \hdots \oplus V(b_\ell)
\end{align} 
is the socle decomposition of two uniserial $\big(\sl(2)\ltimes \mathfrak{h}_n\big)$-modules, where $V$ is standard faithful and $W$ is either of type $Z$ or standard faithful. 
Recall that $V^*\simeq V$ and that the socle decomposition of $W^*$ is $V(b_\ell)\oplus \hdots \oplus V(b_0)$. 

\begin{theorem}
Let $V$ and $W$ be as in \eqref{eq.VV} and \eqref{eq.WW} with 
$V$ standard faithful and $W$ either of type $Z$ or standard faithful. Then $\text{Hom}_{\g}(V,W)$ and $\text{Hom}_{\g}(W,V)$ are zero except in the following cases.
\begin{enumerate}[(1)]
\item $W$ is of type $Z$ and $a_0=b_0$. 
Here $\text{Hom}_{\g}(V,W)$ is $1$-dimensional and it is described by the following arrow
\begin{align*}
V= V(a_0) \oplus V(a_1)\, \oplus\, & V(a_0) \\
 & \quad\downarrow \\
 W=\; & V(b_0) \oplus \cdots \oplus V(b_\ell). 
\end{align*}

\item  $W$ is of type $Z$ and $a_0=b_{\ell}$. 
Here $\text{Hom}_{\g}(W,V)$ is $1$-dimensional
and it is described by the following arrow
\begin{align*}
W= V(b_0) \oplus \cdots \oplus & V(b_\ell) \\
 & \quad\downarrow \\
V=\; & V(a_0) \oplus V(a_1)\, \oplus\, V(a_0).
\end{align*}

\item  $W$ is standard faithful and $a_0=b_0$.
Here $\text{Hom}_{\g}(V,W)$ is $1$-dimensional if 
$V\not\simeq W$ and it is $2$-dimensional if 
$V\simeq W$.
In both cases, a $\g$-module morphism that is not an isomorphism is described by the following arrow
\begin{align*}
V= V(a_0) \oplus V(a_1)\, \oplus\, & V(a_0) \\
 & \quad\downarrow \\
 W=\; & V(b_0) \oplus V(b_1) \oplus V(b_0). 
\end{align*}

\item $n=1$, $a_0=b_1$, $V=FU_{a_0}^-$, $W=Z(b_0,\ell)$, $\ell\ge1$.
Here $\text{Hom}_{\g}(V,W)$ is $1$-dimensional it is described by the following arrow
\begin{align*}
V= V(a_0) \oplus\, V(&a_0-1)\,\oplus V(a_0) \\
 & \quad\downarrow \hspace{1.2cm}\downarrow \\
 W=\; & V(b_0)\;\; \oplus\; V(b_1) \oplus \cdots \oplus V(b_\ell).
\end{align*}

\item $n=1$, $a_0=b_{\ell-1}$, $V=FU_{a_0}^-$, $W=Z(b_\ell,\ell)^*$, $\ell\ge1$.
Here $\text{Hom}_{\g}(W,V)$ is $1$-dimensional it is described by the following arrow
\begin{align*}
 W=\; V(b_0)\oplus \cdots \oplus\; & V(b_{\ell-1}) \oplus V(b_\ell) \\
 & \quad\downarrow \hspace{1.2cm}\downarrow \\
V= &V(a_0) \oplus\, V(a_0-1)\,\oplus V(a_0).
\end{align*}
 \item $n=1$, $a_0=b_1$, $V=FU_{a_0}^+$, $W=Z(b_\ell,\ell)^*$, $\ell\ge1$.
Here $\text{Hom}_{\g}(V,W)$ is $1$-dimensional it is described by the following arrow
\begin{align*}
V= V(a_0) \oplus\, V(&a_0+1)\,\oplus V(a_0) \\
 & \quad\downarrow \hspace{1.2cm}\downarrow \\
 W=\; & V(b_0)\;\; \oplus\; V(b_1) \oplus \cdots \oplus V(b_\ell). 
\end{align*}

\item $n=1$, $a_0=b_{\ell-1}$, $V=FU_{a_0}^+$, $W=Z(b_0,\ell)$, $\ell\ge1$.
Here $\text{Hom}_{\g}(W,V)$ is $1$-dimensional it is described by the following arrow
\begin{align*}
 W=\; V(b_0)\oplus \cdots \oplus\; & V(b_{\ell-1}) \oplus V(b_\ell) \\
 & \quad\downarrow \hspace{1.2cm}\downarrow \\
V= &V(a_0) \oplus\, V(a_0+1)\,\oplus V(a_0).
\end{align*} 

 \item $n>1$, $a_0=0,1$, $V=FU_{(a_0,a_0+m,a_0)}$, $ W=Z(a_0,1)^*$.
Here $\text{Hom}_{\g}(V,W)$ is $1$-dimensional it is described by the following arrow
\begin{align*}
V= V(a_0) \oplus\, V(&a_0+m)\,\oplus V(a_0) \\
 & \quad\downarrow \hspace{1.7cm}\downarrow \\
 W=\; & V(a_0+m) \oplus V(a_0). 
\end{align*}

\item $n>1$, $a_0=0,1$, $V=FU_{(a_0,a_0+m,a_0)}$, $W=Z(a_0,1)$.
Here $\text{Hom}_{\g}(W,V)$ is $1$-dimensional it is described by the following arrow
\begin{align*}
 W=\;  & V(a_0) \oplus V(a_0+m) \\
 & \quad\downarrow \hspace{1.5cm}\downarrow \\
V= &V(a_0) \oplus\, V(a_0+m)\,\oplus V(a_0).
\end{align*} 

\item $n=1$, $a_0=b_1$, $V=FU_{a_0}^-$, $W=FU_{b_0}^+$.
Here $\text{Hom}_{\g}(V,W)$ is $1$-dimensional it is described by the following arrow
\begin{align*}
V= V(a_0) \oplus\, V(&a_0-1)\oplus V(a_0) \\
 & \quad\downarrow \hspace{1.3cm}\downarrow \\
 W=\; & V(b_0) \oplus V(b_0+1)\oplus V(b_0) .
\end{align*}

\item $n=1$, $a_0=b_1$, $V=FU_{a_0}^+$, $W=FU_{b_0}^-$.
\begin{align*}
V= V(a_0) \oplus\, V(&a_0+1)\oplus V(a_0) \\
 & \quad\downarrow \hspace{1.3cm}\downarrow \\
 W=\; & V(b_0) \oplus V(b_0-1)\oplus V(b_0) .
\end{align*}

\end{enumerate}
\end{theorem}

\begin{proof} We will obtain items (1)-(11) by considering
all the cases in which either 
 $S_t(V^*,W)^{\sl(2)}\ne0$ or 
  $S_t(W^*,V)^{\sl(2)}\ne0$, 
 for $t=0,1,2$.

\begin{enumerate}[$\bullet$]
\item Cases where $S_0(V^*,W)^{\sl(2)}\ne0$ or 
  $S_0(W^*,V)^{\sl(2)}\ne0$ and $W$ of type $Z$.

It follows from Theorem \ref{thm.fielvsnofiel} that 
$S_0(V^*,W)^{\sl(2)}\ne0$
 if and only if 
 $a_0=b_0$ (and equal to $a_2$). 
Also, Theorem \ref{thm.fielvsnofiel} implies that, in these cases, 
we have $\dim S_0(V^*,W)^{\sl(2)}=1$ and $S_t(V^*,W)^{\sl(2)}=0$, $t=1,2$.
Hence, $\text{Hom}_{\g}(V,W)$ is $1$-dimensional and it is described by the following arrow
\begin{align*}
V= V(a_0) \oplus V(a_1)\, \oplus\, & V(a_0) \\
 & \quad\downarrow \\
 W=\; & V(b_0) \oplus \cdots \oplus V(b_\ell). 
\end{align*}

Similarly, from Theorem \ref{thm.fielvsnofiel} we know that 
$S_0(W^*,V)^{\sl(2)}\ne0$
 if and only if $a_0=b_\ell$.
 In these cases $S_t(W^*,V)^{\sl(2)}=0$ for $t=1,2$, 
$\text{Hom}_{\g}(W,V)$ is $1$-dimensional and 
it is described by the following arrow
\begin{align*}
W= V(b_0) \oplus \cdots \oplus & V(b_\ell) \\
 & \quad\downarrow \\
V=\; & V(a_0) \oplus V(a_1)\, \oplus\, V(a_0).
\end{align*}

Items (1) and (2) are consequence of this case.

\medskip

\item  Cases where $S_0(V^*,W)^{\sl(2)}\ne0$ or 
  $S_2(V^*,W)^{\sl(2)}\ne0$ and $W$ standard faithful.

It follows from Theorem \ref{thm.fielvsfiel} that 
$S_0(V,W)^{\sl(2)}\ne0$
 if and only if 
 $a_0=b_0$.
 Hence, in what follows, $a_0=a_2=b_0=b_2$.
 In this case, 
 $S_1(V,W)^{\sl(2)}=0$ and 
 $\dim S_0(V,W)^{\sl(2)}=1$.
 Note that if $V\simeq W$, that is 
 $a_1=b_1$, then Theorem \ref{thm.fielvsfiel} says that, in addition, we have
 $\dim S_2(V,W)^{\sl(2)}=1$. 
The non-zero $\g$-morphism corresponding to the 1-dimensional space 
$S_0(V,W)^{\sl(2)}$
is described by the following arrow
\begin{align*}
V= V(a_0) \oplus V(a_1)\, \oplus\, & V(a_0) \\
 & \quad\downarrow \\
 W=\; & V(b_0) \oplus V(b_1) \oplus V(b_0). 
\end{align*}
This is clearly not an isomorphism.
In the particular case when $V\simeq W$, the isomorphism
is the $\g$-morphism corresponding to the 1-dimensional space 
$S_2(V,W)^{\sl(2)}$ (see also Proposition \ref{prop:Zocalo_t}). 
Thus
\[
\dim\text{Hom}_{\g}(V,W)=\begin{cases}
1, & \text{if $V\not\simeq W$ (that is $a_1\ne b_1$);} \\
2, & \text{if $V\simeq W$ (that is $a_1= b_1$)}.
\end{cases}
\] 
Since here $V$ and $W$ are of the same type, we do not need to consider 
$\text{Hom}_{\g}(W,V)$.
This case yields item (3).  

\medskip

\item Cases  where $S_1(V^*,W)^{\sl(2)}\ne0$ or 
  $S_1(W^*,V)^{\sl(2)}\ne0$. 

It follows from Theorems \ref{thm.fielvsnofiel} and \ref{thm.fielvsfiel} that, in all these cases, these spaces are 1-dimensional and 
$S_t(V^*,W)^{\sl(2)}=0$ and $S_t(W^*,V)^{\sl(2)}=0$
for $t=0,2$.
Hence, $\text{Hom}_{\g}(V,W)$ (or $\text{Hom}_{\g}(W,V)$) is $1$-dimensional. 
We describe all these cases below:

\medskip

\begin{enumerate}[$\circ$]
\item Cases with $W$ of type $Z$ and $\text{Hom}_{\g}(V,W)\ne0$.

\begin{enumerate}
\item $n=1$, $a_0=b_1$, $V=FU_{a_0}^-$, $W=Z(b_0,\ell)$, $\ell\ge1$.
\begin{align*}
V= V(a_0) \oplus\, V(&a_0-1)\,\oplus V(a_0) \\
 & \quad\downarrow \hspace{1.2cm}\downarrow \\
 W=\; & V(b_0)\;\; \oplus\; V(b_1) \oplus \cdots \oplus V(b_\ell).
\end{align*}

 \item $n=1$, $a_0=b_1$, $V=FU_{a_0}^+$, $W=Z(b_\ell,\ell)^*$, $\ell\ge1$.
\begin{align*}
V= V(a_0) \oplus\, V(&a_0+1)\,\oplus V(a_0) \\
 & \quad\downarrow \hspace{1.2cm}\downarrow \\
 W=\; & V(b_0)\;\; \oplus\; V(b_1) \oplus \cdots \oplus V(b_\ell). 
\end{align*}

 \item $n>1$, $a_0=0,1$, $ V=FU_{(a_0,a_0+m,a_0)}$, $ W=Z(a_0,1)^*$.
\begin{align*}
V= V(a_0) \oplus\, V(&a_0+m)\,\oplus V(a_0) \\
 & \quad\downarrow \hspace{1.2cm}\downarrow \\
 W=\; & V(a_0+m)\;\; \oplus\; V(a_0). 
\end{align*}
\end{enumerate}
These cases correspond to items (4), (6) and (8).

\medskip

\item Cases with $W$ of type $Z$ and $\text{Hom}_{\g}(W,V)\ne0$.

\begin{enumerate}
\item $n=1$, $a_0=b_{\ell-1}$, $V=FU_{a_0}^+$, $W=Z(b_0,\ell)$, $\ell\ge1$.
\begin{align*}
 W=\; V(b_0)\oplus \cdots \oplus\; & V(b_{\ell-1}) \oplus V(b_\ell) \\
 & \quad\downarrow \hspace{1.2cm}\downarrow \\
V= &V(a_0) \oplus\, V(a_0+1)\,\oplus V(a_0).
\end{align*} 
\item $n=1$, $a_0=b_{\ell-1}$, $V=FU_{a_0}^-$, $W=Z(b_\ell,\ell)^*$, $\ell\ge1$.
\begin{align*}
 W=\; V(b_0)\oplus \cdots \oplus\; & V(b_{\ell-1}) \oplus V(b_\ell) \\
 & \quad\downarrow \hspace{1.2cm}\downarrow \\
V= &V(a_0) \oplus\, V(a_0-1)\,\oplus V(a_0).
\end{align*}
\item $n>1$, $V=FU_{(a_0,a_0+m,a_0)}$, $a_0=0,1$, $W=Z(a_0,1)$.
\begin{align*}
 W=\;  & V(a_0) \oplus V(a_0+m) \\
 & \quad\downarrow \hspace{1.5cm}\downarrow \\
V= &V(a_0) \oplus\, V(a_0+m)\,\oplus V(a_0).
\end{align*} 
\end{enumerate}
These cases correspond to items (5), (7) and (9).

\medskip

\item Cases with $W$ standard faithful and $\text{Hom}_{\g}(V,W)\ne0$.

\begin{enumerate}
\item $n=1$, $a_0=b_1$, $V=FU_{a_0}^-$, $W=FU_{b_0}^+$.
\begin{align*}
V= V(a_0) \oplus\, V(&a_0-1)\oplus V(a_0) \\
 & \quad\downarrow \hspace{1.3cm}\downarrow \\
 W=\; & V(b_0) \oplus V(b_0+1)\oplus V(b_0) .
\end{align*}

\item $n=1$, $a_0=b_1$, $V=FU_{a_0}^+$, $W=FU_{b_0}^-$.
\begin{align*}
V= V(a_0) \oplus\, V(&a_0+1)\oplus V(a_0) \\
 & \quad\downarrow \hspace{1.3cm}\downarrow \\
 W=\; & V(b_0) \oplus V(b_0-1)\oplus V(b_0) .
\end{align*}
\end{enumerate}
These cases correspond to items (10) and (11).
\end{enumerate}
\end{enumerate}
This completes the proof.
\end{proof}

\section{Proof of Theorem \ref{prop.conjecture}}\label{sec.proof_conj}
Let $\mu$ be a possible highest weight in $S_1$.
We start with some general considerations and next we will 
work out each case. 

We know, from Proposition \ref{prop:Zocalo_t},
that $\mu$ must be a highest weight in both 
$V(a)\otimes V(d)$ and $V(b)\otimes V(c)$, that is 
\begin{equation}\label{eq:cotas_mu}
|a-d|,|b-c|\le \mu \le a+d,b+c
\end{equation}
and $\mu\equiv a+d\equiv b+c\mod 2$.
We also know that $\mu$ is indeed a
 highest weight in $S_1$ if and only if
 there is a linear combination 
\[
u_0=q_1v_0^{a,d,\mu}+q_2 v_0^{b,c,\mu},
\]
with $q_1,q_2\ne 0$ (see item (i) in Proposition \ref{prop:Zocalo_t}), that is 
annihilated by $e_s$ for all 
$s=0,\dots,m$. Indeed, this implies that $u_0$ is also annihilated by $z$ and
thus in $S_1$ (see \eqref{eq.soc=invariants}).

We now describe 
$e_sv_0^{a,d,\mu}$ and 
$e_sv_0^{b,c,\mu}$.

On the one hand we have (see \eqref{eq.Vc_en_tensor})
\begin{equation}
v_0^{a,d,\mu}
=\sum_{i,j} 
CG(\tfrac{a}{2},\tfrac{a}{2}-i;\,\tfrac{d}{2},\tfrac{d}{2}-j
\,|\,\tfrac{\mu}{2},\tfrac{\mu}{2})\,
 v_i^a\otimes v_j^d 
\end{equation}
and thus (see \eqref{eq.actionV(m)})
\begin{align}
e_sv_0^{a,d,\mu}
& =\sum_{i,j} 
CG(\tfrac{a}{2},\tfrac{a}{2}-i;\,\tfrac{d}{2},\tfrac{d}{2}-j
\,|\,\tfrac{\mu}{2},\tfrac{\mu}{2})\, v_{i}^a\otimes e_sv_{j}^d \notag \\
& =\sum_{i,j,k} (-1)^j
CG(\tfrac{a}{2},\tfrac{a}{2}-i;\,\tfrac{d}{2},\tfrac{d}{2}-j
\,|\,\tfrac{\mu}{2},\tfrac{\mu}{2})\notag \\
&\hspace{4cm}\times 
CG(\tfrac{c}{2},\tfrac{c}{2}-k;\,\tfrac{d}{2},-\tfrac{d}{2}+j
\,|\,\tfrac{m}{2},\tfrac{m}{2}-s)\,
v_{i}^a\otimes v_{k}^{c} \notag \\
& =\sum_{i,j,k} (-1)^k
CG(\tfrac{a}{2},\tfrac{a}{2}-i;\,\tfrac{d}{2},\tfrac{d}{2}-k
\,|\,\tfrac{\mu}{2},\tfrac{\mu}{2})\notag \\\label{eq:esvad}
&\hspace{4cm}\times 
CG(\tfrac{c}{2},\tfrac{c}{2}-j;\,\tfrac{d}{2},-\tfrac{d}{2}+k
\,|\,\tfrac{m}{2},\tfrac{m}{2}-s)\,
v_{i}^a\otimes v_{j}^{c}.
\end{align}
In this sum, if the coefficient of $v_{i}^a\otimes v_{j}^{c}$ is not zero then we must have
\begin{equation}\label{eq.Coef_no_nulo_ad}
\begin{split}
\frac{a}{2}-i+\frac{d}{2}-k & = \frac{\mu}{2}, \\
\frac{c}{2}-j-\frac{d}{2}+k & =\frac{m}{2}-s.
\end{split}
\end{equation}

On the other hand, we have (see \eqref{eq.Vc_en_tensor})
\begin{equation}
v_0^{b,c,\mu}
 =\sum_{i,j} 
CG(\tfrac{b}{2},\tfrac{b}{2}-i;\,\tfrac{c}{2},\tfrac{c}{2}-j
\,|\,\tfrac{\mu}{2},\tfrac{\mu}{2})\,
 v_i^b\otimes v_j^c
\end{equation}
and thus (see \eqref{eq.actionV(m)})
\begin{align}
e_sv_0^{b,c,\mu}
& =\sum_{i,j} 
CG(\tfrac{b}{2},\tfrac{b}{2}-i;\,\tfrac{c}{2},\tfrac{c}{2}-j
\,|\,\tfrac{\mu}{2},\tfrac{\mu}{2})\,
 e_s v_i^b\otimes v_j^c. \notag \\
& =\sum_{i,j,k} (-1)^i
CG(\tfrac{b}{2},\tfrac{b}{2}-i;\,\tfrac{c}{2},\tfrac{c}{2}-j
\,|\,\tfrac{\mu}{2},\tfrac{\mu}{2})\notag \\
&\hspace{4cm}\times 
CG(\tfrac{a}{2},\tfrac{a}{2}-k;\,\tfrac{b}{2},-\tfrac{b}{2}+i
\,|\,\tfrac{m}{2},\tfrac{m}{2}-s)\,
v_k^a\otimes v_j^c \notag \\
& =\sum_{i,j,k} (-1)^k 
CG(\tfrac{b}{2},\tfrac{b}{2}-k;\,\tfrac{c}{2},\tfrac{c}{2}-j
\,|\,\tfrac{\mu}{2},\tfrac{\mu}{2})\notag \\\label{eq:esvbc}
&\hspace{4cm}\times 
CG(\tfrac{a}{2},\tfrac{a}{2}-i;\,\tfrac{b}{2},-\tfrac{b}{2}+k
\,|\,\tfrac{m}{2},\tfrac{m}{2}-s)\,
v_i^a\otimes v_j^c.
\end{align}
In this sum, if the coefficient of $v_{i}^a\otimes v_{j}^{c}$ is not zero then we must have
\begin{equation}\label{eq.Coef_no_nulo_bc}
\begin{split}
\frac{b}{2}-k+\frac{c}{2}-j & = \frac{\mu}{2}, \\
\frac{a}{2}-i-\frac{b}{2}+k & =\frac{m}{2}-s.
\end{split}
\end{equation}
Either \eqref{eq.Coef_no_nulo_ad} or \eqref{eq.Coef_no_nulo_bc} imply
\begin{equation}\label{eq.imasj}
i+j=\frac{a+c-m-\mu}{2}+s,
\end{equation}
(recall that $0\le i\le a$ and $0\le j \le c$).
Now we consider all the cases.

\bigskip

\noindent
\emph{(i) The case $V_2=E(c,d)$ with $c+d=m$ and $0<a\le c$:} Here
\[
\mu=d+a-2p,\;\; 0\le p\le \min\{a,d\}
\]
and it follows from \eqref{eq.imasj}
that
\begin{equation}
0\le i+j=p-d+s.
\end{equation}
The sum \eqref{eq:esvad} is 
\begin{align*}
e_sv_0^{a,d,\mu}
& =\sum_{i,j,k} (-1)^k
CG(\tfrac{a}{2},\tfrac{a}{2}-i;\,
\tfrac{d}{2},\tfrac{d}{2}-k
\,|\,\tfrac{a+d}{2}-p,\tfrac{a+d}{2}-p)\, \\ 
&\hspace{2cm}\times
CG(\tfrac{m-d}{2},\tfrac{m-d}{2}-j;\,
\tfrac{d}{2},-\tfrac{d}{2}+k
\,|\,\tfrac{m}{2},\tfrac{m}{2}-s)\;\;
v_{i}^a\otimes v_{j}^{c}. \\
\end{align*}
In this sum, it follows from \eqref{eq.Coef_no_nulo_ad} that
\begin{align*}
k&=p-i\\
j&=p-d+s-i.
\end{align*}
The conditions $k\ge 0$ and $0\le j\le m-d$ imply $p-m+s\le i\le \min\{p,p-d+s\}$. Therefore, we obtain (see \eqref{eq.extremoCG0} and \eqref{eq.extremoCG})
\begin{align*}
e_sv_0^{a,d,\mu}
& =\sum_{i=\max\{0,p-m+s\}}^{\min\{p,p-d+s\}} (-1)^{p-i}
CG(\tfrac{a}{2},\tfrac{a}{2}-i;\,
\tfrac{d}{2},\tfrac{d}{2}-p+i
\,|\,\tfrac{a+d}{2}-p,\tfrac{a+d}{2}-p)\, \\ 
&\hspace{1cm}\times
CG(\tfrac{m-d}{2},\tfrac{m-d}{2}-p+d-s+i;\,
\tfrac{d}{2},-\tfrac{d}{2}+p-i
\,|\,\tfrac{m}{2},\tfrac{m}{2}-s)\;\;
v_{i}^a\otimes v_{p-d+s-i}^{c} \\[3mm]
& =\sum_{i=\max\{0,p-m+s\}}^{\min\{p,p-d+s\}} (-1)^{p}
\sqrt{
\frac{
(a+d-2p+1)!\;p!\;(a-i)!\,(d-p+i)! 
}{
(a-p)!\; (d-p)!\; (a+d-p+1) !\; i!\; (p-i)!
}
} \\ 
&\hspace{1cm}\times
\sqrt{
\frac{
(m-s)!\;s!\;(m-d)!\;d!\;
}{
m!\;(m-p-s+i)!\;(p-i)!\;(d-p+i)!\;(p-d+s-i)!
}
} \;\;
v_{i}^a\otimes v_{p-d+s-i}^{c}.
\end{align*}
Thus, 
\[
e_sv_0^{a,d,\mu}=(-1)^{p} \;\sqrt{
\frac{
(m-d)!\;(a+d-2p+1)!\;p!\;d!\;(m-s)!\;s!
}{
(a-p)!\; (d-p)!\; (a+d-p+1) !\; m!
}
} \; w_s^{a,d,\mu}
\]
with 
\[
w_s^{a,d,\mu}=
\sum_{i=\max\{0,p-m+s\}}^{\min\{p,p-d+s\}}
\sqrt{
\frac{
(a-i)!
}{
i!\; (p-i)!^2\;(m-p-s+i)!\; \;(p-d+s-i)!
}
} \;\;
v_{i}^a\otimes v_{p-d+s-i}^{c}.
\]

On the other hand, the sum \eqref{eq:esvbc} is
\begin{align*}
e_sv_0^{b,c,\mu}
& =\sum_{i,j,k} (-1)^k 
CG(\tfrac{m-a}{2},\tfrac{m-a}{2}-k;\,
\tfrac{m-d}{2},\tfrac{m-d}{2}-j
\,|\,\tfrac{a+d}{2}-p,\tfrac{a+d}{2}-p)\,  \\ 
&\hspace{2cm}\times
CG(\tfrac{a}{2},\tfrac{a}{2}-i;\,
\tfrac{m-a}{2},-\tfrac{m-a}{2}+k
\,|\,\tfrac{m}{2},\tfrac{m}{2}-s)\;\;
v_i^a\otimes v_j^c.
\end{align*}
In this sum, it follows from \eqref{eq.Coef_no_nulo_bc} that
\begin{align*}
j&=p-d+s-i\\
k&=m-a-s+i
\end{align*}
and the condition $ k\ge 0$ implies $i\ge a-m+s$, and condition $j\geq 0$ implies $p-d+s\ge i$. 
Therefore, we obtain (see \eqref{eq.extremoCG0} and \eqref{eq.extremoCG})
\begin{align*}
e_sv_0^{b,c,\mu}
& =\sum_{i=\max\{0,a-m+s\}}^{\min\{a,p-d+s\}} (-1)^{m-a-s+i}
 \\ 
&\hspace{0.8cm}\times
CG(\tfrac{m-a}{2},-\tfrac{m-a}{2}+s-i;\,
\tfrac{m-d}{2},\tfrac{m-d}{2}-p+d-s+i
\,|\,\tfrac{a+d}{2}-p,\tfrac{a+d}{2}-p)\,  \\ 
&\hspace{0.8cm}\times
CG(\tfrac{a}{2},\tfrac{a}{2}-i;\,
\tfrac{m-a}{2},\tfrac{m-a}{2}-s+i
\,|\,\tfrac{m}{2},\tfrac{m}{2}-s)\;\;
v_i^a\otimes v_{p-d+s-i}^c \\[3mm]
& =\sum_{i=\max\{0,a-m+s\}}^{\min\{a,p-d+s\}}
\sqrt{
\frac{
(a+d-2p+1)!\,(p+m-a-d)!\,(s-i)!\,(m-p-s+i)!
}{
(d-p)!\,(a-p)!\,(m-a-s+i)!\,(p-d+s-i)!\,(m-p+1)!
}
} \\ 
&\hspace{2cm}\times
\sqrt{
\frac{ 
s!\;(m-s)!\,a!\,(m-a)!\,
}{
m!\,(a-i)!\,(m-a-s+i)!\;i!\;(s-i)!
}
} \;\;
v_i^a\otimes v_{a-p+s-i}^c.
\end{align*}
Thus
\[
e_sv_0^{b,c,\mu}=
\sqrt{
\frac{
(m-s)!\;s!\;
(a+d-2p+1)!\; (p+m-a-d)!\; a!\; (m-a)!
}{
(d-p)!\; (a-p)!\;(m-p+1)!\; m!
}
}\; w_s^{b,c,\mu}
\]
where
\[
w_s^{b,c,\mu}=
\sum_{i=\max\{0,a-m+s\}}^{\min\{a,p-d+s\}} 
\sqrt{
\frac{
(m-p-s+i)!
}{
i!\; (m-a-s+i)!^2\;(p-d+s-i)!\; (a-i)!\; 
}
} \;\;
v_{i}^a\otimes v_{p-d+s-i}^{c}.
\]

\

Now, if $a\le d$ and $p=a$, we have, for all $0\le s\le m$,
\begin{align*}
w_s^{a,d,\mu} & =\sum_{i=\max\{0,a-m+s\}}^{\min\{a,a-d+s\}}
\sqrt{
\frac{
1
}{
i!\; (a-i)!\;(m-a-s+i)!\; \;(a-d+s-i)!
}
} \;\;
v_{i}^a\otimes v_{p-d+s-i}^{c}\\
& = w_s^{b,c,\mu}.
\end{align*}
This shows that
\[u_0=(-1)^a\;\sqrt{d+1}\;v_0^{a,d,\mu}-\sqrt{b+1}\;v_0^{b,c,\mu}\]
is, indeed, a highest weight vector, of weight $\mu=d-a=b-c$, in $S_1$.

\

If $p < a$ then, for $s=d$, the sum defining $w^{a,d,\mu}_d$ has the index $i$ running up to $i=a$ while 
the sum defining 
$w_d^{b,c,\mu}$ has the index $i$ running only 
up to $i=p$. In both cases, 
all the coefficients are non-zero, and thus 
$\{w^{a,d,\mu}_1, w^{b,c,\mu}_1\}$ is linearly independent. 
This shows that there is no possible $\mu$ in $S_1$ and thus $S_1=0$.
This completes the proof in this case.

\bigskip

\noindent \emph{(ii) The case $V_2=E(c,d)$ with $d=c+m$:} 
Since $a<m\le d$ and by equation \eqref{eq:cotas_mu} we have 
\begin{align*}
\mu& = b+c-2p, \;\; 0\leq p\leq \min\{c,b\},\\
\mu& = a+d-2p', \;\; 0\leq p'\leq a.
\end{align*}
This implies $p'-p=a$ and hence $p'=a$ and $p=0$. 
This yields 
\[
\mu=b+c=d-a.
\]
First we prove that, if $c=0$, then $S_1(V_1,V_2)\simeq V(b)$.
In this case, \eqref{eq:esvad} becomes 
\begin{align*}
e_sv_0^{a,d,\mu}
& =\sum_{i,k} (-1)^k
CG(\tfrac{m-b}{2},\tfrac{m-b}{2}-i;\,
\tfrac{m}{2},\tfrac{m}{2}-k
\,|\,\tfrac{b}{2},\tfrac{b}{2})\, \\ 
&\hspace{2cm}\times
CG(0,0;\,
\tfrac{m}{2},-\tfrac{m}{2}+k
\,|\,\tfrac{m}{2},\tfrac{m}{2}-s)\;\;
v_{i}^a\otimes v_{0}^{0} 
\end{align*}
(the index $j$ is $0$). 
It follows from \eqref{eq.Coef_no_nulo_ad} that
\begin{align*}
k&=m-s\\
i&=-b+s.
\end{align*}
Therefore, $e_sv_0^{a,d,\mu}=0$ if $s<b$ and, 
for $s\geq b$ we have (see \eqref{eq.extremoCG0} and \eqref{eq.extremoCG11})
\begin{align*}
e_sv_0^{a,d,\mu}
& =(-1)^{m-s}
CG(\tfrac{m-b}{2},-\tfrac{m-b}{2}+m-s;\,
\tfrac{m}{2},-\tfrac{m}{2}+s
\,|\,\tfrac{b}{2},\tfrac{b}{2})\, \\ 
&\hspace{3cm}\times
CG(0,0;\,
\tfrac{m}{2},\tfrac{m}{2}-s
\,|\,\tfrac{m}{2},\tfrac{m}{2}-s)\;\;
v_{s-b}^a\otimes v_{0}^{0}\\[3mm]
&=
\sqrt{\frac{(b+1)\,(m-b)!\,s!}{(m+1)!\,(s-b)!}}\;\;
v_{s-b}^a\otimes v_{0}^{0}.
\end{align*}

On the other hand, since $c=0$ and $\mu=b$, \eqref{eq:esvbc} becomes (see also \eqref{eq.actionV(m)} or \eqref{eq.actionV(m)1})
\begin{align*}
e_sv_0^{b,c,\mu}
& =e_sv_0^b\otimes v_0^0 \\
& =
\begin{cases}
CG(\tfrac{a}{2},\,\tfrac{a}{2}-(s-b);\,
\tfrac{b}{2},-\tfrac{b}{2}
\,|\,\tfrac{m}{2},\tfrac{m}{2}-s)\;\;
v_{s-b}^a\otimes v_0^0, & \text{if $s\ge b$};\\[2mm]
0, & \text{if $s<b$}.
\end{cases} \\
& =
\begin{cases}
\sqrt{\frac{a!\,s!}{(s-b)!\,m!}}\;\;
v_{s-b}^a\otimes v_{0}^0, & \text{if $s\ge b$};\\[2mm]
0, & \text{if $s<b$}.
\end{cases}
\end{align*}
This shows that
 \[
 u_0=\sqrt{b+1}\;v_0^{a,d,\mu}-\sqrt{m+1}\;v_0^{c,d,\mu} 
  \]
 is, indeed, a highest weight vector of weight $\mu=b$, in $S_1$.

\

Now, suppose that $c\neq 0$ and set $s=m$. Recall that 
$\mu=b+c=d-a$. 
When we consider the sum \eqref{eq:esvad}, it follows from \eqref{eq.Coef_no_nulo_ad} that 
\begin{align*}
j=k&=a-i.
\end{align*}
The condition $0\le j\le c$ implies $a-c\le i \le a$ and hence \eqref{eq:esvad} becomes (see \eqref{eq.extremoCG11})
\begin{align*}
e_mv_0^{a,d,\mu}&=\sum_{i=\max\{0,a-c\}}^a(-1)^{a-i}\, CG(\tfrac{a}{2},\tfrac{a}{2}-i;\,\tfrac{d}{2},\tfrac{d}{2}-(a-i)\,|\,\tfrac{d-a}{2},\,\tfrac{d-a}{2})\notag\\
&\hspace{3cm} \times
CG(\tfrac{c}{2},\tfrac{c}{2}-(a-i);\,\tfrac{d}{2}, -\tfrac{d}{2}+a-i\,|\,\tfrac{m}{2},-\tfrac{m}{2})\,v_i^a\otimes v_{a-i}^{c}\\[3mm]
&= \sum_{i=\max\{0,a-c\}}^a(-1)^{i+d-a}\sqrt{\frac{(c+b+1)\;(m-b)!\;(c+b+i)!}{(c+m+1)!\;i!}}\\
&\hspace{4cm}
\times \sqrt{\frac{(m+1)\;c!\;(c+b+i)!}{(c+m+1)!\;(c-m+b+i)!}}\;\;v_i^a\otimes v_{a-i}^{c}. 
\end{align*}

On the other hand, when we consider the sum 
\eqref{eq:esvbc},
 it follows from \eqref{eq.Coef_no_nulo_bc} that
\[
j=a-i=-k
\]
and the condition $k\geq 0$ implies that $k=j=0$ and $i=a=m-b$. 
Thus, \eqref{eq:esvbc} is
\begin{align*}
e_mv_0^{b,c,\mu}&=CG(\tfrac{b}{2},\tfrac{b}{2};\, \tfrac{c}{2},\tfrac{c}{2}\,|\, \tfrac{c+b}{2},\tfrac{c+b}{2}) 
\;CG(\tfrac{m-b}{2}, -\tfrac{m-b}{2};\, \tfrac{b}{2},-\tfrac{b}{2}\,|\, \tfrac{m}{2}, -\tfrac{m}{2})\;\; v_a^a\otimes v_0^{c}\\[3mm]
&=v_a^a\otimes v_0^{c}.
\end{align*}

Since $c\neq 0$, the sum in $e_mv_0^{a,d,\mu}$ has at least two non-zero terms, while the sum in $e_mv_0^{c,d,\mu}$ has a single non-zero term, and thus $\{e_mv_0^{a,d,\mu},e_mv_0^{c,b,\mu}\}$ is linearly independent. This completes the proof in this case.

\bigskip

\noindent \emph{(iii) The case $V_2=E(c,d)$ with $c=d+m$:} 
 Since $b<m\le c$, it follows from \eqref{eq:cotas_mu} that
\begin{align*}
\mu& = b+c-2p, \;\; 0\leq p\leq b\\
\mu& = a+d-2p', \;\; 0\leq p'\leq \min\{a,d\}.
\end{align*}
This implies $p-p'=b$ and hence the only option is 
$p=b$, $p'=0$ and this yields 
\[
\mu=a+d=c-b.
\] 
We compute now 
$e_sv_0^{a,d,\mu}$.
It follows from \eqref{eq.imasj} and \eqref{eq.Coef_no_nulo_ad} that
\begin{align*}
k&=-i\\
j&=s-i,
\end{align*}
and since $k\geq 0$, we have $k=i=0$ and $j=s$. 
Therefore, 
\eqref{eq:esvad} becomes (see also 
 \eqref{eq.extremoCG0} and \eqref{eq.extremoCG1})
 \begin{align*}
e_sv_0^{a,d,\mu}
& =
CG(\tfrac{a}{2},\tfrac{a}{2};\,
\tfrac{d}{2},\tfrac{d}{2}
\,|\,\tfrac{a+d}{2},\tfrac{a+d}{2})\;
CG(\tfrac{d+m}{2},\tfrac{d+m}{2}-s;\,
\tfrac{d}{2},-\tfrac{d}{2}
\,|\,\tfrac{m}{2},\tfrac{m}{2}-s)\;\;
v_{0}^a\otimes v_{s}^{c}\\[3mm]
&=\sqrt{\tfrac{(d+m-s)!\;(m+1)!}{(d+m+1)!\;(m-s)!}}\;\;v_0^a\otimes v_s^c.
\end{align*}
As always, the reader should check that all the numbers under the factorial sign are non-negative.

We compute now 
$e_sv_0^{b,c,\mu}$.
It follows from \eqref{eq.imasj} and
\eqref{eq.Coef_no_nulo_bc} that
\begin{align*}
j&=s-i\\
k&=m-a-s+i
\end{align*}
and conditions $k\ge 0$ and $j\ge 0$ imply $s\geq i\geq s-(m-a)$. Thus, it follows from \eqref{eq:esvbc}, \eqref{eq.extremoCG11} and \eqref{eq.extremoCG0} that 
\begin{align*}
e_sv_0^{b,c,\mu}
& = \sum_{i=\max\{0,s-(m-a)\}}^{\min\{a,s\}}
CG(\tfrac{m-a}{2},-\tfrac{m-a}{2}+s-i;\,
\tfrac{d+m}{2},\tfrac{d+m}{2}-(s-i)
\,|\,\tfrac{a+d}{2},\tfrac{a+d}{2})  \\ 
&\hspace{3.1cm}\times
CG(\tfrac{a}{2},\tfrac{a}{2}-i;\,
\tfrac{m-a}{2},\tfrac{m-a}{2}-s+i
\,|\,\tfrac{m}{2},\tfrac{m}{2}-s)\;\;
v_{i}^a\otimes v_{s-i}^c \\[3mm]
&=\sum_{i=\max\{0,s-(m-a)\}}^{\min\{a,s\}}(-1)^{s-i}\sqrt{\tfrac{(d+m-s+i)!\;(m-a)!\;(a+d+1)}{(d+m+1)!\;(m-a-s+i)!}}\\
&\hspace{4.7cm}
\times \sqrt{\tfrac{a!\;(m-a)!\;(m-s)!\;s!}{i!\;(s-i)!\;m!\;(a-i)!\;(m-a-s+i)!}}\;\;v_{i}^a\otimes v_{s-i}^c
\end{align*}
Again, note that all the numbers under the factorial sign are non-negative.

For $s=m-a=b$, the sum giving $e_sv_0^{b,c,\mu}$ has the index $i$ running from $i=0$ to $i=\min\{a,b\}\neq 0$, while the sum giving $e_sv_0^{a,d,\mu}$ has the index $i$ running only up to $i=0$. In both cases, all the coefficients are non-zero, and thus $\{e_sv_0^{b,c,\mu},e_sv_0^{a,d,\mu}\}$ is linearly independent. This shows that there is no possible $\mu$ in $S_1$ and thus $S_1=0$. This completes this case and the proof of the theorem.

\section{COI Statement}
The authors have no conflict of interest to declare that are relevant to this article.

\bibliographystyle{plain}
\bibliography{bibliografia_tensor-uniserials}

\end{document}